\documentclass[final,times]{elsarticle}

 \makeatletter
\def\ps@pprintTitle{%
 \let\@oddhead\@empty
 \let\@evenhead\@empty
 \def\@oddfoot{\centerline{\thepage}}%
 \let\@evenfoot\@oddfoot}
\makeatother

 \newcommand{\grad}{\triangledown}

\usepackage{soul}
\usepackage{graphicx,enumitem,dsfont}
 \usepackage[dvips]{epsfig}
\usepackage{amscd}
\usepackage{amssymb}
\usepackage{amsthm}
\usepackage{amsmath}
\usepackage{latexsym}

\usepackage{upref}
\usepackage[colorlinks]{hyperref}
\usepackage{color}
\usepackage{yfonts}
\theoremstyle{plain}
\newtheorem{thm}{Theorem}[section]
\theoremstyle{plain}
\newtheorem{lem}[thm]{Lemma}

\theoremstyle{definition}
\newtheorem{defi}{Definition}[section]
\newtheorem{rem}{Remark}[section]
\newtheorem*{maintheorem*}{Main Theorem}
\newtheorem*{maincorollary*}{Main Corollary}
\newenvironment{Assumptions}
{%
\setcounter{enumi}{0}

\begin{enumerate}}%
{\end{enumerate} }

{%
\setcounter{enumi}{0}

\begin{enumerate}}%
{\end{enumerate} }

\newcommand{\R}{\ensuremath{\mathbb{R}}}

\newcommand{\gi}{\textfrak{i}}

\newcommand{\m}{\mathbf{m}}
\newcommand{\bj}{\mathbf{j}}
\newcommand{\bk}{\mathbf{k}}
\newcommand{\Eta}{\eta^{\alpha,\beta}}

\newcommand{\rn}{\mathbb{R}^{N}}
\newcommand{\zn}{\mathbb{Z}^{N}}
	\newcommand{\Z}{\mathbb{Z}}

\newcommand{\A}{\alpha}
\newcommand{\B}{\beta}

\newcommand{\I}{\mathcal{I}}
\newcommand{\J}{\tilde{\mathcal{J}}}

\usepackage[normalem]{ulem}

\numberwithin{equation}{section} \allowdisplaybreaks

\journal{preprint}
\begin{document}

\begin{frontmatter}

\title{On the rate of convergence for monotone numerical schemes for nonlocal Isaacs' equations}

\author[Imran H. Biswas]{Imran H. Biswas}
\address[Imran H. Biswas]{
 Centre for Applicable Mathematics,
 Tata Instiute of Fundamental Research,
  P.O.\ Box 6503,
  Bangalore 560065, India}

\author[Imran H. Biswas]{Indranil Chowdhury}

\author[Espen R. Jakobsen]{Espen R. Jakobsen}
\address[Espen R. Jakobsen]{
Norwegian University of Science and Technology, NO--7491, Trondheim,
Norway}

\begin{abstract}
We study monotone numerical schemes for nonlocal Isaacs equations, the
dynamic programming equations of stochastic differential games 
with jump-diffusion state processes. These equations are
fully-nonlinear non-convex equations of order less
than $2$. In this paper they are also allowed to be degenerate 
and have non-smooth solutions. The main contribution is a
series of new a priori error estimates:
The first results 
for {\em nonlocal} Isaacs equations, the first
  general results for {\em degenerate} non-convex equations of order
  greater than $1$, and the first results  in the viscosity solution
  setting giving the {\em precise dependence} on the fractional order of the
  equation. We also observe a new phenomena,
    that the rates differ when the nonlocal diffusion coefficient depend on
    $x$ and $t$, only on $x$, or on neither. 
\end{abstract}


\end{frontmatter}


\section{Introduction}

 In this paper we obtain error estimates for monotone approximation
 schemes for nonlocal Isaacs-Bellman equations originating from
 optimal stochastic control and differential game theory:
 \begin{align}\label{eqn:main}
 & u_t + \inf_{\A\in \mathcal{A}}\sup_{\B\in\mathcal{B}} \left\{ -f^{\A,\B}(t,x) +c^{\A,\B}(t,x) u(t,x) - b^{\A,\B}(t,x) \cdot \grad u(t,x) - \I^{\A,\B}[u](t,x) \right\}  = 0  \mbox{\quad in} \ Q_T, \\
\label{int-cond}&u(0,x)  = u_0(x)  \mbox{\quad in} ~ \rn,   
 \end{align}where $\mathcal{I}^{\alpha,\beta}$ is a nonlocal operator defined by
\begin{align} \label{n_local:term}
\I^{\A,\B}[\phi](t,x) := \int_{\mathbb{R}^M\backslash\{0\}} \Big( \phi(t,x+\eta^{\A,\B}(t,x; z)) - \phi(t, x) -\eta^{\A,\B}(t,x;z)\cdot \nabla_x\phi(t,x)\Big) \, \nu(dz)
\end{align}
for smooth bounded functions $\phi$. Here  $Q_T := (0,T] \times \rn $,  $\mathcal{A}$  and
   $\mathcal{B}$ are metric spaces, $f^{\A,\B},c^{\A,\B},b^{\A,\B},\eta^{\A,\B}$ are  $ \R , \R, \rn$, and $ \rn $
 valued functions respectively, while the L\'{e}vy measure $\nu$ is a
 nonnegative Radon measure satisfying the L\'evy integrability assumption
 \ref{A4} in Section \ref{preliminary}.

The diffusion part of this equation ($\mathcal
I^{\alpha,\beta}$) is purely nonlocal, and under the assumptions of Section \ref{preliminary},
$\mathcal{I}^{\alpha,\beta}$ is a non-positive 
fractional differential operator of order $\sigma\in[0,2)$. The
  fractional Laplacian $-(-\Delta)^{\frac\sigma2}$ is not
  covered, but all similar operators coming from tempered or truncated
  processes are. In particular almost all non-local operators
  appearing in finance are included \cite{CTbook}.
 In general 
  this equation is a fully nonlinear, non-convex, nonlocal PDE (an
  integro-PDE) that may have any order $\sigma\in[0,2)$. In particular,
    it may have order greater than one. Moreover, since we
  also allow the equations to be degenerate, solutions are typically not
  smooth. Under Lipschitz type regularity assumptions on the 
  coefficients and data, the problems are
  well-posed in the viscosity solution sense \cite{CIL} having 
  merely H\"older or Lipschitz continuous solutions. First and fractional
  derivatives need not exist. The precise assumptions and results can
  be found in Section \ref{preliminary}.
  The literature on viscosity solutions and nonlocal PDEs is by now very
  large,
  but the results we will need here are mainly covered by
  \cite{Jakobsen:2005jy, BI:P07} and the references therein.

The study of Isaacs and Bellman equations is
primarily motivated by their connection with the theory of stochastic
differential games and stochastic control. Equations of the form
\eqref{eqn:main} appear  when the state dynamics is given by a controlled L\'{e}vy-It\^{o} type SDE driven by
a pure-jump L\'{e}vy process. By the dynamic programming principle (DPP), the value
functions of such games satisfy nonlocal PDEs of the form
\eqref{eqn:main}. These equations are called the Isaacs or DPP equations
for the differential games.  We refer to
\cite{Biswas2012,Evans:1984fv,Fleming:1989cg} for more on
differential games and dynamic programming equations. Note that
if $\eta^{\A, \B}\equiv 0$ or $\nu\equiv0$, then there is no diffusion
and \eqref{eqn:main} becomes the widely studied first order Isaacs equation
corresponding to a deterministic game (see e.g. \cite{Evans:1984fv}).
If the state process is driven by a Brownian motion, 
 then the related DPP equation is a second order PDE
 (cf. \cite{Fleming:1989cg}). This case will not be considered here.



 The numerical approximations we consider here are monotone finite-difference
 quadrature methods in the spirit of e.g. \cite{BJK1}. We refer to
 \eqref{eq:the scheme} in Section \ref{sec:main} for the precise form
 of these approximations. 
  The main contribution of this paper is a series of new and very accurate
error estimates in this setting. 
If solutions are Lipschitz continuous, then 
these estimates may take the form 
\begin{align}
\label{estim1}\left\|U - u \right\|_{L^\infty(Q_T)} \leq C_T\begin{cases}
   \Delta t^{\frac{1}{2}}+  \Delta x^{\frac12} & \text{if}\ \sigma\in[0,1),\\
   \Delta t^{\frac{1}{2}}+  \Delta x^{\frac12}|\ln\Delta x| & \text{if}\ \sigma=1,\\
   \Delta t^{\frac{1}{2}}+  \Delta x^{\frac{2-\sigma}{2\sigma}}& \text{if}\ \sigma\in(1,2),
  \end{cases}
\end{align}
where $\sigma\in[0,2)$ is the order of the nonlocal term, and $\Delta
t>0,\Delta x>0$ are time and space grid parameters.
    In general solutions are only
  H\"older continuous in time, and then also the rates in time may depend on
  $\sigma$.
Surprisingly, we also discover a
  new phenomenon. When $\sigma\in(1,2)$, the convergence rates differ
  depending on whether $\eta$ depends on $(x,t)$, only on $x$, or on
  neither! We find in Remark \ref{optd-rem} that 
 \begin{align}\label{estim2}
 \left\|U - u \right\|_{L^\infty(Q_T)} \leq C
\left\{  
		\begin{array}{ll}
		 (\Delta t)^{\frac{2-\sigma}{2\sigma}} + ( \Delta
                  x)^{\frac{2-\sigma}{2\sigma}} & \quad \mbox{when
                    $\eta$ depends on $x,t$}, \\[0.2cm]
		 (\Delta t)^{\frac{1}{2\sigma}} + ( \Delta
                  x)^{\frac{2-\sigma}{2\sigma}} & \quad \mbox{when
                    $\eta$ only depends on $x$}, \\[0.2cm]
		 (\Delta t)^{\frac{1}{2\sigma}} + ( \Delta
                  x)^{\frac{2-\sigma}{2}} & \quad  \mbox{when 
                    $\eta$ does not depend on $x,t$}.
		\end{array}
\right.
 \end{align}
Precise statements and results are given in Section \ref{sec:main}.

   The study of numerical approximation in the context of viscosity
   solutions began in the early eighties with pioneering papers of
   Lions, Crandall and others. In some of the early papers
   \cite{CDF, CL:2approx, Lions:2010hs, souganidis1985}, the authors 
   obtained a priori error estimates for consistent  monotone  schemes for
   first order HJB equations. These results are  derived through
   suitable modifications of the viscosity solution uniqueness proofs for
   the corresponding equations. These arguments can not be extended to
   2nd order equations, and it took more than a decade before a
   solution was found by N. V. Krylov. In a series of
   articles \cite{Krylov:1997nf,Krylov:2000yy,Krylov:2005lj},
   Krylov introduced the method of shaking the coefficients and was
   able to establish error estimates for a class of
   monotone schemes for convex second order HJB equations. These
   results were then extended and complemented by Barles $\&$ Jakobsen
   in \cite{BJ2002, BJ2005, Barles:2006jf}. In all of these papers, and
   the many others building upon them,
   convexity and a type of Jensen's inequality is crucial.

   For non-convex equations like the Isaacs equation, there are no
   general results giving error estimates for numerical
   methods. However, in special cases there are some results: In one
   space dimension \cite{J2002}, for special types of non-convex
   equations \cite{Zidani:2006, J2006}, and for uniformly elliptic/parabolic
   equations \cite{Caffarelii:2008gf,krylov:2015,T2015,T2015a}.
In the two first cases the proofs rely on the special structure of the
problems (one dimension and not too non-convex) and are not suitable
for general equations/dimensions, while in the last case it relies on
some type of elliptic regularity. This last direction of research was
initiated by Cafferelli and Souganidis in \cite{Caffarelii:2008gf} (but
see also \cite{krylov:2015}), where they obtain an (unknown) algebraic rate
of convergence for equations with rather general non-convex nonlinearities.  
In spite of all these results, it seems that the problem is very far
from understood in the case of general, possibly degenerate, Isaacs equations.
     
     The story of nonlocal Bellman-Isaacs equations is a more recent
     one and there is already a significant literature
     addressing the issues of numerical approximations and the
     related error analysis.  Most of the development in this
     direction have taken place in the last ten years,  see
     e.g. \cite{BJK2,BJK1,JKL08} for general error estimates for 
     convex and nonlocal HJB equations. These results are extensions
     of the results for local 2nd order equations
     (Krylov-Barles-Jakobsen type theory) and convexity is again
     crucial. For non-convex nonlocal problems there are no
     results on error estimates as far as we know.

     At this point, we note that convexity is not playing any role in
     the proof of the error estimates for first order equations. But,
     as we have already mentioned, these techniques do not work for 2nd order
     problems. However, for a different class of equations and weak
     solution concept (nonlinear convection-diffusion equations and entropy
     solutions), it was noticed in \cite{CJ2011} that first  
     order error estimation techniques surprisingly could work also for
     nonlocal/fractional 
     problems of any order less than 2. At least for certain natural numerical
     approximations. Is it possible to do similar things also for the
     nonlocal Isaacs equations \eqref{eqn:main} and in a viscosity
     solution setting? 
     The goal of this paper is to investigate if, and
     to what extent, we can extend first order error estimation
     techniques to nonlocal Isaacs equations \eqref{eqn:main} of
     any order less than two.

Because of the nonlocal term, the analysis necessarily becomes much
more involved than in the first order case, and it leads (as usual) to
3 different cases: (i) The supercritical case where
$\sigma\in[0,1)$ and drift/convection dominates, (ii) the critical
  case $\sigma=1$ where drift and diffusion is in balance, and (iii)
  the subcritical case where $\sigma\in(1,2)$ and diffusion
  dominates. In this paper we give precise and rigorous error
  estimates in all cases, cf. e.g. \eqref{estim1} and
  \eqref{estim2}. In case (i) we get the same (and hence the optimal)
  rate as for first 
  order equations \cite{CDF, CL:2approx, souganidis1985}.  In case (ii)
  we get a rate with a logarithm, and in case (iii) we find a rate
  depending on $\sigma$. Under certain conditions these rates are
  consistent with the rates in \cite{CJ2011}.
  Note that the rates go to $0$ when $\sigma\to2$. This behaviour is
  correct and is an artifact of the  
  numerical method. Under our low regularity assumptions, these
  results are the best possible results for this method.
In case (iii) (cf. \eqref{estim2}) we also observe that the rates differ according to whether $\eta$
  depend on $x$ and $t$, only on $x$, or on neither of them.  This is
    a new phenomenon that is not present for local equations.
To summarize, the main novelties of this paper are:
\begin{enumerate}
\item The first error estimates for numerical schemes for {\em
  nonlocal} Isaacs equations.
\item The first error bounds for  general {\em degenerate} non-convex
  equations of order greater than $1$.
\item The first error bounds for a numerical scheme in the viscosity solution setting giving the
  {\em precise dependence} of the order $\sigma$ of the nonlocal term.
\item The first error bounds where the rates depend on whether the
  jump term $\eta$  depend on $(x,t)$, only on $x$, or on neither. 
\end{enumerate}
As a part of our effort to get precise estimates
correctly reflecting the fractional order $\sigma$ of the nonlocal
term, we also prove a new and refined time regularity result for
viscosity solutions.


  The rest of the paper is organised as follows. 
       In Section
       \ref{preliminary}, we list the assumptions and state the
       wellposedness result and a priori estimates for
       \eqref{eqn:main}--\eqref{int-cond}, including the new and more
       accurate time regularity result. In Section
       \ref{sec:main}, we introduce the schemes, establish properties
       such as wellposedness, consistency, monotonicity and stability, and 
       state our main results, the error estimates. The proof of the
       these estimates are given in Section
       \ref{proof}. In Section
         \ref{sec:higher-order}, the last
       section of the paper, we explain how our techniques can be used
       to obtain error estimates for a larger class of monotone
       approximations of \eqref{eqn:main}. But this extension comes at
       a price, the rates for more accurate schemes will be suboptimal.
  
\section{Preliminaries } \label{preliminary}
In this section we state our main assumptions, define the relevant
concept of solutions -- viscosity solutions, and state and partially
prove a wellposedness result for \eqref{eqn:main}-\eqref{int-cond}.
We start with some notation. By $C, K$ we mean various constants which
may change from line to line. 
The Euclidean norm on any $\mathbb{R}^d$-type space is
denoted by $|\cdot|$. For any subset $Q\subset \mathbb{R}\times
\mathbb{R}^N$ and for any bounded, possibly vector valued,
function on $Q$, we define the following norms,
\begin{align*}
	&\|w\|_0 := \sup_{(t,x)\in Q} |w(t,x)|,\\
	&\|w\|_{1} := \|w\|_0 + \sup_{(t,x)\neq(s,y)}
	\frac{|w(t,x)-w(s,y)|}{|t-s|+|x-y|}. 
\end{align*}
Note that if $w$ is independent of $t$, then $\|w\|_1$ is the
Lipschitz (or $W^{1,\infty}$) norm of $w$. We use $C_b(Q)$ to denote the
space of bounded continuous real valued functions on $Q$.
We use the notation $h$ to denote the vector $(\Delta t, \Delta x )$ involving the mesh parameters,
and any dependence on $\Delta t$, $\Delta x$ will be denoted by
subscript $h$. The grid is denoted by $\mathcal G_h$ and is a subset
of $\bar Q_T$ which need not be uniform or even discrete in
general. We also set $\mathcal G_h^0=\mathcal G_h\cap\{t=0\}$
and $\mathcal G_h^+=\mathcal G_h\cap\{t>0\}$.

       We now list the working assumptions of this paper.
       These  are sufficient for the wellposedness and regularity
       results for \eqref{eqn:main}--\eqref{int-cond}.

 \begin{Assumptions}
\item\label{A1} 
 The sets $\mathcal{A} , \mathcal{B}$ are separable metric
 spaces, $c^{\A,\B}(t,x) \geq0$, and $c^{\A,\B}(t,x),f^{\A,\B}(t,x),b^{\A,\B}(t,x)$ and\\ $\eta^{\alpha, \beta}(t,x;z)$ are continuous in $\A$, $\B$, $t,x$  and $z$. 
\item \label{A2} There exists a constant $K>0$ such that for every $\A,\B$ , 
$$ \|u_0\|_{1} + \|f^{\A,\B}\|_{1}+ \|c^{\A,\B}\|_{1}+ \|b^{\A,\B}\|_{1} \leq K . $$

\item \label{A3}  For $x,y \in \rn$, $\A\in \mathcal{A},\B \in
  \mathcal{B}$ and $z\in \mathbb{R}^M $, there is a function
  $\rho(z)\geq0$ such that
\begin{align*}
|\Eta(t,x;z) - \Eta(s,y; z)| \leq  \rho(z)\,(|x-y|+ |t-s|) \quad \mbox{and} \quad |\Eta(t,x;z)| \leq \rho(z)
\end{align*}
and 
$$|\rho(z)|\leq K|z|\quad \text{for}\quad |z|<1\qquad\text{and}\qquad
1\leq  \rho(z) \leq \rho(z)^2\quad \text{for} \quad |z|>1.$$
  
  \item\label{A4} The L\'{e}vy measure $\nu$  is a nonnegative Radon measure on $\big(\R^M, \mathcal{B}(\R^M)\big)$ satisfying
  \begin{align*}
   \int_{|z|<1} |z|^2\nu(\,dz)+ \int_{|z|>1} \rho(z)^2\nu(\,dz) < \infty. 
  \end{align*}

\item\label{A5} There is a $\sigma \in (0, 2)$, a constant $C>0$, and
  density $k(z)$ of $\nu(dz)$ for $|z|<1$ satisfying
  \begin{align*}
  0\le k(z) \le \frac{C}{|z|^{M+\sigma}}\qquad\text{for}\qquad|z|<1.
\end{align*} 

\end{Assumptions}

\begin{rem}
(a) Typical examples are $\eta=\bar\eta(x)z$ and $\eta=\bar\eta(x)(e^z-1)$,
and for $\nu$,
$$\nu(dz)=\frac{c_\sigma
  e^{-K|z|}dz}{|z|^{N+\sigma}}\quad\text{and} \quad \nu(dz)=1_{|z|<1}\frac{c_\sigma
 dz}{|z|^{N+\sigma}}$$ 
for $\sigma\in(0,2)$, i.e. tempered or truncated $\sigma$-stable
L\'{e}vy measures.  Near $z=0$ these L\'{e}vy measures behave as the L\'{e}vy measure
associated to the fractional Laplacian $(-\Delta)^{\sigma/2}$, and
their (pseudo-differential) orders is $\sigma$ as it is for
$(-\Delta)^{\sigma/2}$. We will see that we get different estimates
when $\sigma<1$, $\sigma=1$, or $\sigma>1$.

\smallskip

\noindent(b) Assumptions \ref{A3}, \ref{A4}, and \ref{A5}
are quite general and encompass most models from finance
\cite{CTbook}, and under \ref{A3} and \ref{A4} there is a standard
viscosity solution theory for \eqref{eqn:main}. Note that assumption
\ref{A5} only requires an upper bound on the density. This bound is
needed to get an explicit convergence rate. 
\smallskip

\noindent (c) All assumptions can be relaxed in such a way that our
techniques and results would still apply: \ref{A3} and 
\ref{A5} can be replaced by more general integral conditions like
$\int |\eta(t,x;z)-\eta(s,y;z)|^2\nu(dz)\leq 
L(|x-y|+|t-s|)$, $\int_{|z|<r}|\eta(t,x;z)|^2\nu(dz)\leq
Kr^{2-\sigma}$, etc., and \ref{A4} can be relaxed when it comes to the
integrability at infinity and absolute continuity. This is somehow
straight forward, but we omit it since it would obscure the message
and make the paper much longer and more technical.
\end{rem}  

 We now give the definition of viscosity solution for
  \eqref{eqn:main}-\eqref{int-cond}. To this end, we define
\begin{align} 
\I^{\A,\B}_{\kappa}[\phi](t,x) = \int_{B(0,\kappa)} \left(\phi(t, x + \eta^{\A,\B}(t, x;z))- \phi(t,x)-\eta^{\A,\B}(t,x;z)\cdot\nabla_x \phi(t,x)   \right) \ \nu(dz) \ , \notag \\
 \label{viscosity:term2}
\I^{\A,\B,\kappa}[u; p](t,x) = \int_{\R^M \setminus B(0,\kappa)} \left(u(t, x + \eta^{\A,\B}(t,x;z))- u(t,x) -\eta^{\A,\B}(t,x;z)\cdot p  \right) \ \ \nu(dz)\ , 
\end{align}for $(\A, \B)\in \mathcal{A}\times \mathcal{B}$, $\kappa\in
(0,1)$, $\phi\in C^2$, and bounded semicontinuous functions
$u$. By \ref{A3}--\ref{A4}, $\I^{\A,\B,\kappa}[u; p]$ and
$\I^{\A,\B}_{\kappa}[\phi]$  are well-defined, in the first case since
$\int_{|z|>\kappa}\nu(dz)<\infty$  and in the second case since
$$|\I^{\A,\B}_{\kappa}[\phi](x,t)|\leq
\frac12\|D^2\phi(\cdot,t)\|_{L^\infty(B(x,\kappa))}
\int_{|z|<\kappa}K^2\rho(z)^2\nu(dz)<\infty.
$$

\begin{defi}
(i) A function $u \in USC_b(Q_T)$ is a viscosity subsolution of
\eqref{eqn:main} if for any $k\in(0,1)$, $\phi \in C^2(Q_T)$, and
global  maximum point $(t,x)\in Q_T$ of $u - \phi$,
\begin{align*}
& \phi_t(t,x) + \inf_{\A}\sup_{\B} \Big\{ -f^{\A,\B}(t,x) +c^{\A,\B}(t,x) u(t,x) - b^{\A,\B}(t,x) . \grad \phi(t,x) \\
&\hspace{7cm}- \I^{\A,\B}_k[\phi](t,x) -\I^{\A,\B,k}[u,\nabla_x\phi(t,x)](t,x)  \Big\}  \leq 0.
\end{align*}
\noindent (ii) \quad  A function $v \in LSC_b(Q_T)$ is a viscosity supersolution of
\eqref{eqn:main} if for any $k\in(0,1)$, $\psi \in C^2(Q_T)$, global minimum point $(t,x)\in Q_T$ of $v - \psi$,
\begin{align*}
& \psi_t(t,x) + \inf_{\A}\sup_{\B} \Big\{- f^{\A,\B}(t,x) +c^{\A,\B}(t,x) v(t,x) - b^{\A,\B}(t,x) . \grad \psi(t,x) \\&\hspace{7cm}-\I^{\A,\B}_k[\psi](t,x) -\I^{\A,\B,k}[v, \nabla_x\psi(t,x)](t,x)  \Big\}  \geq 0.
\end{align*}
\noindent (iii)\quad A function $w \in C_b(Q_T)$ is a viscosity solution of \eqref{eqn:main} if it is both a sub and supersolution. 
\end{defi}

We then have the following wellposedness and Lipschitz/H\"{o}lder regularity results for
        \eqref{eqn:main}.

\begin{thm}\label{wellpos_1}
Assume \ref{A1}--\ref{A4} hold.

\smallskip
\noindent (a) If  $u$ and $v$ are respectively viscosity sub and
supersolutions of
\eqref{eqn:main} with $u(0,\cdot)\le v(0,\cdot)$, then $u\le v$.
\smallskip

\noindent (b) There exists a unique bounded viscosity solution $u$ of
the initial value problem \eqref{eqn:main}--\eqref{int-cond}.
\smallskip

  \noindent (c) There is a constant $K\geq0$ such that the solution
  $u$ from (b) satisfies for all $x,y\in\R^N$, $t,s\in[0,T]$,
\begin{align}\label{bomega}|u(x,t)-u(y,s)|\leq
  K\Big(|x-y|+\bar\omega(t-s)\Big)\qquad\text{where}\qquad \bar\omega(r):=\begin{cases}|r| &\text{if }
\sigma\in[0,1),\\ |r|(1+|\ln r|)&\text{if }
  \sigma=1,\\ |r|^{\frac1\sigma}&\text{if } \sigma\in(1,2).\end{cases}
  \end{align}
\medskip

\noindent (d) Assume in addition 
\begin{align*}
  K(u_0):= \sup_{\A,\B}\big\|{\mathcal I}^{\alpha,\beta}[u_0]
  \big\|_{L^{\infty}([0,T]\times \rn)}<\infty.
\end{align*}
\noindent   Then there is $C\geq0$ depending only on the
data \ref{A1}--\ref{A5} such that the solution $u$ from (b)
satisfies for all $x,y\in\R^N$, $t,s\in[0,T]$,
$$|u(x,t)-u(y,s)|\leq
 C\Big(|x-y|+(1+K(u_0))|t-s|\Big).$$

\end{thm}

The wellposedness and $x$-regularity results are quite standard, but the time
regularity results are new and more precise than earlier
results. These time regularity results are somewhat parallel to the results in Lemma
5.4 in \cite{CJ2011}, but the equation, norm and solution concepts are
different.

\begin{rem}
Under assumptions \ref{A3} and \ref{A4}, either (i) $v \in W^{2,
  \infty}(\rn)$, or (ii) $v \in W^{1, \infty}(\rn)$ and \ref{A5} holds with $\sigma<1$, are
sufficient conditions for $K(v)<\infty$. See Lemma \ref{K-lem} below.
\end{rem}

In the proof of Theorem \ref{wellpos_1} we will need the following lemma.

\begin{lem}\label{K-lem}
Assume \ref{A3}--\ref{A5}. Then there is a
constant $C>0$ such that for all $\phi\in C^2_b(\R^N)$ and $\epsilon\in(0,1)$,
$$K(\phi)\leq
\begin{cases}C\Big(\epsilon^{2-\sigma}\|D^2\phi\|_0+(1+\epsilon^{1-\sigma})\|D\phi\|_0\Big),
& \text{if }
\sigma\in(1,2),\\[0.2cm] C\Big(\epsilon\|D^2\phi\|_0+(1+|\ln\epsilon|)\|D\phi\|_0\Big),& 
\text{if }\sigma=1,\\[0.2cm]
C\|D\phi\|_0, &\text{if }\sigma\in[0,1).
\end{cases} $$

\end{lem}
\begin{proof}
  When $\sigma<1$, then $|{\mathcal
    I}^{\alpha,\beta}\phi(x)|\leq C\|D\phi\|_0\int
  |\eta^{\alpha,\beta}(t,x,z)|\,\nu(dz)$. Since
  $\int\eta^{\alpha,\beta}\nu(dz)\leq \int
  \rho(z)\nu(dz)<\infty$ by \ref{A3} and \ref{A4}, the bound on
  $K(\phi)$ follows by taking the supremum over $x,\alpha,\beta$. 
  For $\sigma\geq1$, we split the integral in three parts and use Taylor's
  theorem: 
  \begin{align*}
 {\mathcal I}^{\alpha,\beta}[\phi]&=  \int \Big( \phi(x+\eta)-\phi(x)-\eta \nabla\phi(x)\Big)
 \ \nu(dz) \\ 
  & = 
  \int_{|z|<\epsilon}\int_0^1(1-t)\eta^T D^2\phi(x+t\eta)\eta
  \   dt\,\nu(dz) +
  \Big(\int_{\epsilon\leq |z|<1}+\int_{ |z|\geq1}\Big)\int_0^1\Big(\nabla \phi(x+t\eta)-\nabla \phi(x)\Big)\eta
  \, dt \,\nu(dz).
    \end{align*}
  By assumption \ref{A3} -- \ref{A5}, it follows that
  $${\mathcal I}^{\alpha,\beta}[\phi]\leq C\|D^2\phi\|_0\int_{|z|<\epsilon}|z|^2\frac{dz}{|z|^{N+\sigma}}+ C\|D\phi\|_0\bigg(\int_{\epsilon<|z|<1}|z|\frac{dz}{|z|^{N+\sigma}}+\int_{|z|\geq1}\rho(z)\nu(dz)\bigg).$$
By \ref{A3} and \ref{A4}, the last integral is finite, and the result
then follows from computing the two first integrals in polar
coordinates and taking the supremum over $x,\alpha,\beta$. 
\end{proof}

    \begin{proof}[Proof of Theorem \ref{wellpos_1}]
     We refer to Theorem 3.1 of the article \cite{Jakobsen:2005jy} for
     a proof of part (a) and $x$-regularity part of (c) and (d). Part (b) then
     follows e.g. from Perron's 
     method \cite{BJK_SW}. Time regularity in part (c) and (d) is
     new. We start by proving (d) and then use this result to prove (c).
     \medskip

\noindent     (d) First we show Lipschitz in 
     time at $t=0$ by using the comparison principle and the fact that
     $w^{\pm}(t,x) = u_0(x) \pm Ct$ are super- and subsolutions of
     \eqref{eqn:main} if $C$ is large enough. To see this, insert
     $w^{\pm}$ into the equation and use the regularity of $u_0$ to
     conclude. Here the assumption $K(u_0)<\infty$ is crucial and 
      minimal. To get Lipschitz regularity for all times, we use a 
     continuous dependence result and the $t$-Lipschitz regularity of
     the coefficients. See Theorem 5.1 and Theorem 5.3
     of  \cite{Jakobsen:2005jy} for the details, and note that there
     is no growth in $x$ of the estimates here since the coefficients
     and solutions are bounded.    
     \medskip
     
\noindent (c)       
 Let $0<\epsilon<1$ and regularize (by mollification) the initial data to get 
$u_0^\epsilon\in C_b^\infty(\R^N)$ satisfying
 $\|D^ku_0^\epsilon\|_0\leq  C\epsilon^{1-k}$ and
 $\|u_0-u_0^\epsilon\|_0\leq\epsilon$ (since $u_0$ is Lipschitz). Then let $u^\epsilon$ be the
 corresponding solution of \eqref{eqn:main}--\eqref{int-cond}. By (a)
 again $|u-u^\epsilon|\leq  
\|u_0^\epsilon-u_0\|_0\leq C\epsilon$, and by the estimates on
$D^ku_0^\epsilon$ and Lemma \ref{K-lem} with $\phi=u_0^\epsilon$,
\begin{align}\label{K-eps}K(u_0^{\epsilon})\leq C\begin{cases} 
1& \text{if } \sigma\in[0,1),\\
  (1+|\ln\epsilon|) & \text{if } \sigma=1,\\
  \epsilon^{1-\sigma}& \text{if } \sigma\in(1,2).\\
  \end{cases}
  \end{align}
By part (d) we have that $|u^\epsilon(t,x)-u_0^\epsilon(x)|\leq 
C(1+K(u_0^{\epsilon}))t$, and by the triangle inequality
$$|u(t,x)-u_0(x)|\leq C(\epsilon+K(u_0^{\epsilon})t+\epsilon).$$
When $\sigma<1$, $\sigma=1$, and $\sigma>1$, we take $\epsilon=0$,
$\epsilon=t$, and $\epsilon=t^{\frac1\sigma}$ respectively. 
This proves the result for $s=0$, $t\in[0,1]$ (and $x=y$). The result
trivially holds for
$s=0,t>1$, since then e.g. $|u(x,t)-u(x,0)|\leq
2\|u\|_0t^{\frac1\sigma}$. The general result then 
follows from the $t$-Lipschitz regularity of the coefficients and the same
continuous dependence result as in part (d). 
\end{proof}      
       

  \section{The main results: Error estimates for a monotone scheme}
  \label{sec:main}

In this section, we introduce a natural monotone difference-quadrature
scheme for \eqref{eqn:main}. The time discretizations include
explicit, implicit and explicit-implicit schemes. For these schemes we
prove wellposedness, $L^\infty$-stability, and the main results,
several estimates on their rates of convergence in $L^\infty$.

\medskip

For simplicity we consider a uniform
grid in space and time. For $M>0$, let $\Delta x >0$ and $\Delta
t:=\frac{T}M$ be the discretization parameters/mesh size in the time
and space and $h=\big(\Delta t, \Delta x \big)$. The corresponding
mesh is
       \begin{align*}
       \mathcal{G}_h^N = \big\{ (t_n, x_{\m}): t_n =n \Delta t, x_{\m}= \m \, \Delta x;\, \m\in\mathbb{Z}^N, n=0, 1,....,M\big\}.
       \end{align*}
To obtain a full discretization of \eqref{eqn:main}, we follow
\cite{BJK1} and perform the following steps:

\bigskip     
\noindent{\bf Step 1. Approximate singular diffusion by
  bounded diffusion.}
For $\delta \geq\Delta x$ we approximate $\mathcal{I}^{\A,\B}[\phi]$ by
replacing $\nu(dz)$ by the truncated non-singular measure
$\nu_{\delta}(\,dz):=\mathbf{1}_{|z|>\delta }(z)\,\nu(dz)$ in
\eqref{n_local:term}:
   \begin{align*}
    \mathcal{I}^{\A,\B,\delta}[\phi](t,x)&=  \int_{|z|>\delta} \big( \phi(t,x+\eta^{\A,\B}(t,x; z)) - \phi(t, x) -\eta^{\A,\B}(t,x;z)\cdot \nabla_x\phi(t,x)\big) \, \nu(dz)\\
      & = \mathcal{J}^{\A,\B,\delta}[\phi](t,x) - b_{\delta}^{\A,\B}(t,x)\cdot\nabla_x \phi(t,x),
   \end{align*} where
   \begin{align*}
    \mathcal{J}^{\A,\B,\delta}[\phi](t,x) =   \int_{|z|>\delta} \big( \phi(t,x+\eta^{\A,\B}(t,x; z)) - \phi(t, x) \big) \, \nu(dz),\quad b_{\delta}^{\A,\B}(t,x) =\int_{|z|>\delta} \eta^{\A,\B}(t,x;z) \, \nu(dz).
    \end{align*} 
This is a non-singular, nonnegative, consistent approximation of
$\mathcal{I}^{\A,\B}$, and a 
standard argument using Taylor's theorem gives the truncation error
\begin{align}
\label{eq:error-trunc} 
\big|\mathcal{I}^{\A,\B}[\phi]-\mathcal{I}^{\A,\B,\delta}[\phi]\big|
  \le  \frac12\|D^2 \phi\|_0\sup_{x,\alpha,\beta}\int_{|z|<\delta}|\eta^{\alpha,\beta}(t,x;z)|^2\nu(dz)\leq K\delta^{2-\sigma}\|D^2 \phi\|_0\qquad\text{for}\qquad \phi\in C_b^2(\R^N),
\end{align}
where the last inequality follows by \ref{A3}--\ref{A5}. 
Let $\tilde{b}_\delta^{\A,\B}(t,x) :=
b^{\A,\B}(t,x)-b_{\delta}^{\A,\B}(t,x)$. We approximate
\eqref{eqn:main} by replacing 
$\mathcal{I}^{\A,\B}$ by
$\mathcal{I}^{\A,\B,\delta} =\mathcal{J}^{\A,\B,\delta}-
b_{\delta}^{\A,\B}\cdot\nabla$, 
\begin{align}
\label{eqn:main-perturbed-1} & u_t^{\delta} + \inf_{\A\in
                               \mathcal{A}}\sup_{\B\in\mathcal{B}}
                               \left\{ -f^{\A,\B}(t,x) +c^{\A,\B}(t,x)
                               u^\delta(t,x) -
                               \tilde{b}_\delta^{\A,\B}(t,x) . \grad
                               u^\delta(t,x) -
                               \mathcal{J}^{\A,\B,\delta}[u^\delta](t,x)
                               \right\}  = 0  \mbox{\quad in} \ Q_T.
\end{align} 
\smallskip

\noindent{\bf Step 2. Discretize the local drift.}\ \  We discretize $\tilde{b}_\delta^{\alpha, \beta}\cdot\nabla u$ by simple upwind
finite differences: 
\begin{align*}
\mathcal{D}_{h}^{\alpha, \beta,\delta}[u](t,x)  &:=  \sum_{i=1}^N \Big[ \tilde{b}_{\delta,i}^{\A,\B,+}(t,x) \frac{u(t,x+e_i\Delta x)-u(t,x)}{\Delta x }
 + \tilde{b}_{\delta,i}^{\A,\B,-}(t,x) \frac{u(t,x-e_i\Delta x)
                                                                           -u(t,x)}{\Delta x}\Big]\\ \nonumber
 & = \sum_{\bj\neq 0} d_{h, \bj}^{\A, \B,\delta}(t,x) \Big[ u(t, x+x_{\bj})- u(t, x)\Big],
\end{align*}
where $\{e_i\}_i \subset \rn$ is the standard basis of $\rn$,  $
b^{\pm}=\max(\pm b,0)$, $d_{h, \pm  e_i}^{\A, \B,\delta}(t,x) = \frac{\tilde b^{\A,\B,\pm}_{\delta,i}(t,x)}{\Delta
  x}\geq0$ 
and $d_{h,\bj}^{\A, \B,\delta}(t,x)=0$ otherwise.
Hence the discretization is positive/monotone, and it is consistent since 
  \begin{align}
 \label{eq:error-drift} \big|\tilde b^{\alpha,\beta}_\delta(t,x)\cdot\nabla \phi(x)
    -\mathcal{D}_{h}^{\alpha, \beta,\delta}[\phi](t,x)\big| \le
    \frac12\Delta x \sum_{i}\big|\tilde b^{\A,\B}_{\delta,i}(t,x)\big| 
    \|D^2 \phi\|_0\leq K\Delta x\,\Gamma(\sigma,\delta)\|D^2
    \phi\|_0\ \ \text{for}\ \ \phi\in C_b^2(\R^N),
 \end{align}
where
 \begin{align*} 
\Gamma(\sigma, \delta) =
\left\{
       \begin{array}{ll}
       \delta^{1-\sigma} \quad & \text{when} \quad \sigma > 1, \\
       -\log \delta \quad & \text{when} \quad \sigma = 1,\\
       1 \quad & \text{when} \quad \sigma <1.
       \end{array}
\right.
\end{align*} 
The last inequality follows by the definition of $\tilde
b^{\alpha,\beta}_\delta$ since $\int_{|z|>\delta}|\eta^{\alpha,\beta}(t,x;z)|\,\nu(dz)\leq C\Gamma(\sigma, \delta)$ by \ref{A3}--\ref{A5}.

\bigskip     
\noindent{\bf Step 3. Discretize the nonlocal diffusion.} \ \ We
discretize $\mathcal{J}^{\A, \B,\delta}$ by a quadrature formula 
obtained by replacing the integrand by a monotone interpolant
(cf. \cite{BJK1}): 
\begin{align}
 \notag   \mathcal{J}_h^{\A, \B,\delta}[\varphi](t,x) :=
  &\int_{|z|>\delta} \gi_h\big[\varphi(t,x+\cdot)-
    \varphi(t,x)\big](\eta^{\A,\B}(t,x;z))\nu(\,dz),
\end{align}
where $\gi_h$ is piecewise linear/multilinear interpolation on the
spatial grid $\Delta x\mathbb Z^N$. That is,
       \begin{align}\label{interpolation1}
& \gi_h[\phi](x) = \sum_{\bj \in \mathbb{Z}^N} \phi(x_{\bj})
  \omega_{\bj} (x; h)\quad\text{for}\quad x\in \rn,
\end{align} 
where the weights $\omega_{\bj}$ are the standard ``tent functions'' satisfying
$0\leq\omega_{\bj}(x;h)\leq 1$, $\omega_{\bj}(x_{\mathbf{k}};h) =
\delta_{\bj,\mathbf{k}}$, $\sum_{\bj} \omega_{\bj}=1$, 
  $\mathrm{supp}\,\omega_{\bj}\subset B(x_{\bj},2\Delta x)$, and
  $\|D\omega_{\bj}\|_0\leq C(\Delta x)^{-1}$. Note that the sum in
\eqref{interpolation1} is always finite. We can rewrite the
approximation in discrete monotone form:
 \begin{align*}
   \mathcal{J}_h^{\A, \B,\delta}[\varphi](t,x) = 
\sum_{\bj \in \Z^N} \big(\varphi(t,x+x_{\bj})-\varphi(t,x)\big)
   \kappa_{h,\bj}^{\A,\B,\delta}(t,x); \quad  \kappa_{h,\bj}^{\A,\B,\delta}(t,x;h) =\textstyle \int_{|z|>\delta} \omega_{\bj}(\eta^{\A,\B}(t,x;z); h)\nu(dz),
             \end{align*}
where $\kappa_{h,\bj}^{\A,\B,\delta}$ is well-defined and
nonnegative. This approximation is nonnegative, and since
\begin{align}\label{int_err}
\big|\gi_h[\varphi](x)- 
\varphi(x)\big| \le K \|D^2\varphi\|_0 (\Delta x)^2,
\end{align}
it is consistent with truncation error
 \begin{align}
\label{eq:nonlocal-error}  |\mathcal{J}^{\alpha, \beta,\delta}[\phi]
   -\mathcal{J}_h^{\alpha, \beta,\delta}[\phi]| \le K (\Delta x)^2 \|D^2
   \phi\|_0\int_{|z|>\delta} \nu(dz)\leq K_I\|D^2
   \phi\|_0\frac{(\Delta x)^2}{\delta^{\sigma}}\qquad\text{for}\qquad \phi\in C_b^2(\R^N).
 \end{align}
The last inequality follows from \ref{A5}. We also note
that since all $\omega_\bj$'s have same diameter compact support and \ref{A3}
and \ref{A5} hold with $\sigma\in(0,2)$, there is a constant $K_N$
depending only on $N$ such that  
                 \begin{align*}
                    \sum_{\bj \neq 0}  \kappa_{h,\bj}^{\A,\B,\delta}(t,x)
                   \le  \sum_{\bj \neq 0}\|D\omega_\bj\|_0 \int_{|z|>\delta}
                   \big|\eta^{\A,\B}(t,x;z)\big|\,\nu(dz)\leq
                   \frac{K_N}{\Delta x}\Gamma(\sigma,\delta).  
                   \end{align*}

\medskip     
\noindent{\bf Step 4. The full discretization of \eqref{eqn:main}.} \ \ 
Combining the previous steps we obtain the following semidiscrete
approximation of \eqref{eqn:main} (cf. \eqref{eqn:main-perturbed-1}):
 \begin{align}\label{eqn:main-semidiscrete}
 & u_t + \inf_{\A\in \mathcal{A}}\sup_{\B\in\mathcal{B}} \left\{ -f^{\A,\B}(t,x) +c^{\A,\B}(t,x) u(t,x) -\mathcal{D}_h^{\A, \B,\delta}[u](t,x)- \mathcal{J}_h^{\A,\B,\delta}[u](t,x) \right\}  = 0. 
\end{align} 
To discretize in time we use a two-parameter monotone $\theta$-like
method that allows for explicit, implicit, and explicit-implicit
versions (cf. \cite{BJK1}): 
For $\theta, \vartheta\in [0,1]$, 
\begin{align}\label{eq:the scheme}
 \notag U_{\bj} ^n = &\ U_\bj^{n-1}-\Delta t  \inf_{\A\in \mathcal{A}}\sup_{\B\in\mathcal{B}} \left\{- f^{\A,\B, n-1}_\bj + c_\bj^{\A, \B,n} U_\bj^{n-1} -\theta \mathcal{D}_h^{\A, \B,\delta}[U]_\bj^{n}-(1-\theta) \mathcal{D}_h^{\A, \B,\delta}[U]_\bj^{n-1} \right. \\
 \notag & \hspace*{6cm}\left. -\vartheta \mathcal{J}_h^{\A,\B,\delta}[U]_\bj^{n}
 -(1-\vartheta) \mathcal{J}_h^{\A,\B,\delta}[U]_\bj^{n-1} \right\}\\ 
 &\hspace{8cm}\text{for}\quad \bj \in \mathbb{Z}^N,\quad 0\leq n\leq M,\\
 \notag  U_\bj^0 =&\ u(0, x_\bj)\quad\text{for}\quad \bj \in \mathbb{Z}^N,
 \end{align}
where $U^n_{\bj}= U_h(t_n,x_{\bj})$ is the solution of
the scheme and 
$g^{n}_\bj:=g(t_n, x_\bj)$ for any function $g$ and $(t_n,
x_\bj)\in\mathcal{G}^N_h$.  
With this convention, 
 \begin{align*}
\mathcal{D}_h^{\A,
   \B,\delta}[\phi]^n_{\bar{\bj}} = \sum_{\bj\neq 0} d_{h, \bj,
   \bar{\bj}}^{\A, \B,\delta, n} \Big[ \phi(t_n, x_{\bar{\bj}}+x_{\bj})-
   \phi(t_n, x_{\bar{\bj}})\Big] \quad \text{and}\quad
  \mathcal{J}_h^{\A, \B,\delta}[\phi]^n_{\bar{\bj}} = \sum_{\bj\neq 0} \kappa_{h, \bj, \bar{\bj}}^{\A, \B,\delta, n} \Big[ \phi(t_n, x_{\bar{\bj}}+x_\bj)- \phi(t_n, x_{\bar{\bj}})\Big],
\end{align*}
and we may rewrite our scheme \eqref{eq:the scheme} as
\begin{align}
\label{eq:scheme_new_form}\inf_{\alpha\in \mathcal{A}}\sup_{\beta \in \mathcal{B}}\Big\{
a_{\bar{\bj},\mathbf{0}}^{n,n}(\A, \B) U_{\bar{\bj}}^{n} -
\sum_{\bj\neq \mathbf{0}} a_{\bar{\bj}, \bj}^{n,n}(\A, \B)
U_{\bar{\bj}+\bj}^{n} - \sum_{\bj} a_{\bar{\bj}, \bj}^{n,n-1}(\A,
\B) U_{\bar{\bj}+\bj}^{n-1}-\Delta t f_{\bar{\bj}}^{\A, \B, n}\Big\}=
0
\end{align}
with
\begin{align*}
a_{\bar{\bj},\mathbf{0}}^{n,m}(\A, \B)&=\begin{cases} 1 +\Delta t \theta  \sum_{\bj\neq 0}  d_{h,\bj, \bar{\bj}}^{\A, \B,\delta, m} + \Delta t\vartheta \sum_{\bj\neq 0}\kappa_{h,\bj, \bar{\bj}}^{\A, \B,\delta, m}& \text{if}\quad m =n \\
                              1- \Delta t \big[ (1-\theta) \sum_{\bj\neq 0}   d_{h, \bj,\bar{\bj}}^{\A, \B,\delta, m} + (1-\vartheta) \sum_{\bj\neq 0}\kappa_{h,\bj, \bar{\bj}}^{\A, \B,\delta, m}+ c_{\bar{\bj}}^{\A, \B, n}] & \text{if}\quad m =n-1, \end{cases}\\
   a_{\bar{\bj}, \bj}^{n,m}(\A, \B)&=\begin{cases} \Delta t \theta
   d_{h,\bj, \bar{\bj}}^{\A, \B,\delta, m} + \Delta t\vartheta\kappa_{h,\bj,
     \bar{\bj}}^{\A, \B,\delta, m} & \text{if}\quad m =n \\
\Delta t \big[ (1-\theta)  d_{h,\bj, \bar{\bj}}^{\A, \B,\delta, m} + (1-\vartheta)\kappa_{h,\bj, \bar{\bj}}^{\A, \B,\delta, m} ] & \text{if}\quad m =n-1. \end{cases}
\end{align*}
Since $d,\kappa\geq0$, we see that the scheme \eqref{eq:scheme_new_form}  has nonnegative
coefficients and hence is monotone under the CFL
condition:
 \begin{align}
\label{eq:CFL-condtion} \Delta t \Big[ (1-\theta)\sum_{\bj\neq 0}   d_{h, \bj,\bar{\bj}}^{\A, \B,\delta, n-1} + (1-\vartheta)\sum_{\bj\neq 0}\kappa_{h, \bj,\bar{\bj}}^{\A, \B,\delta, n-1} + c_{\bar{\bj}}^{\A, \B, n}]\Big] \le 1.
 \end{align}
By the discussion and definitions of $d$
and $\kappa$ above, for all $0<\Delta x\leq\delta \le 1$,
          \begin{align*} 
        \sum_{\bj\neq 0} d_{h, \bj, \bar{\bj}}^{\A, \B,\delta, n} \le \frac{K_{D}}{\Delta x}\Gamma(\sigma,\delta)\qquad\text{and}\qquad \sum_{\bj\neq 0} \kappa_{h, \bj, \bar{\bj}}^{\A, \B,\delta, n} \le \frac{K_{I}}{\Delta x} \Gamma(\sigma,\delta)
          \end{align*}
for some constants $K_D,K_I$. Hence the CFL condition is
satisfied when
 \begin{align}
 \label{eq:alternate-cfl}\frac{\Delta t}{\Delta x} \Gamma(\sigma,\delta)\Big( (1-\theta)K_D + (1-\vartheta) K_I\Big) + \Delta t \sup_{\A,\B} |c^{\A,\B}|\le 1.
 \end{align}
 \begin{rem}\
   \label{CFL-rem}
\smallskip

\noindent (a) The scheme is explicit when
$\theta=0=\vartheta$, implicit when $\theta=1=\vartheta$,
$\theta$-method like when $\theta=\vartheta$, and explicit-implicit
with explicit convection and implicit diffusion when $\theta=0$
and $\vartheta=1$.
\medskip

\noindent (b) It is possible to use other monotone approximations in steps
1 -- 4, and obtain schemes that can be analyzed using minor
modifications of the arguments we present here.

\medskip

\noindent (c) The CFL condition \eqref{eq:alternate-cfl} gives a
constraint on the relation between $\delta$, $\Delta x$, $\Delta t$
when the scheme is not completely implicit. In the ``first order''
case, when $\sigma\in(0,1)$ in \ref{A5}, we get the usual CFL condition  
$$\Delta t\leq K \Delta x.$$
In the critical case $\sigma=1$, then $\Delta t\leq K \Delta x\,
|\ln\Delta x|$. When the order of the equation is greater than $1$, $\sigma\in(1,2)$ in \ref{A5}, then
$$\Delta t\leq K \delta^{\sigma-1}\Delta x,$$
which when $\delta=(\Delta x)^{\frac1\sigma}$ (giving the optimal
convergence rate, see below) gives
$$\Delta t\leq K \Delta x^{2-\frac1\sigma}.$$
\end{rem}

 We have the following existence, uniqueness and stability result
 for the scheme.

 \begin{thm}\label{thm:existence-scheme}
 Assume \ref{A1}--\ref{A5}, $0<\Delta x\leq \delta\leq 1$, and the CFL condition
 \eqref{eq:CFL-condtion}. 
\medskip

\noindent (a) (Monotone scheme)\ If  $U_h$ and
$V_h$ are bounded sub and 
supersolutions of \eqref{eq:the scheme} with $U_h(0,\cdot)\le V_h(0,\cdot)$, then $U_h\le V_h$.
\medskip

\noindent (b) There exists a unique bounded  solution $U_h$ of
the initial value problem \eqref{eq:the scheme}--\eqref{int-cond}.
\medskip

\noindent (c) ($L^\infty$-stability) \ 
The solution $U_h$ from (b) satisfies 
$|U_h(t_n,x)|\le \|u_0\|_0 +t_n\sup_{\A,\B} \|f^{\A,\B}\|_0.$

  \medskip
    \noindent (d) There is a constant $K\geq0$ such that the solution
  $U_h$ from (b) satisfies for all $x$, $t_n$,
$$|U_h(t_n,x)-u_0(x)|\leq K\bar\omega(t_n),\qquad\text{where
      $\bar\omega$ is defined in \eqref{bomega}}.$$

\noindent (e) Assume in addition that $K(u_0)<\infty$ (cf. Theorem
\ref{wellpos_1}). Then there is $C\geq0$ only depending on the
data \ref{A1}--\ref{A5} such that the solution $U_h$ from
(b) satisfies  for all $x$, $t_n$, 
$$|U_h(t_n,x)-u_0(x)|\leq C (1+K(u_0))t_n.$$

 \end{thm}
 \begin{proof}[Proof of Theorem \ref{thm:existence-scheme}] 
The proofs of (a)--(c) are standard. Part (a) is a direct consequence
of the scheme having positive coefficients, and part (c) follow from 
(a) since $\|u\|_0\pm t_n\sup_{\alpha,\beta}\|f^{\alpha,\beta} \|$ are
super- and subsolutions. Part (b), existence and uniqueness, can
be proved using time-iteration and Banach fixed point theorem. The proof is
essentially the same as the proof of Theorem 3.1 in \cite{BJK1}. Part
(d) and (e) are new and non-standard. We will prove
these results in the same way as for the solution of the continuous
problem \eqref{eqn:main}--\eqref{int-cond}, cf. Theorem \ref{wellpos_1} (c) and
(d). First we prove (e), and then use this result to prove (d).
\medskip

\noindent (e) 
Note that
$V^\pm(x,t_n)=u_0(x)\pm Ct_n$  
are super- and subsolutions of the scheme \eqref{eq:the
  scheme}--\eqref{int-cond} 
if $h$ is sufficiently small and
$$C\geq 1+K(u_0)+
\sup_{\alpha,\beta}\bigg(\|u_0\|_1\|b^{\alpha,\beta}\|_0+\|u_0\|_0\|c^{\alpha,\beta}\|_0+\|f^{\alpha,\beta}\|_0\bigg). $$
The result then follows since $V^-\leq U_h\leq V^+$ by comparison
(part (a)).
\medskip

\noindent (d) We regularize (by mollification) the initial data to get 
$u_0^\epsilon$ and let $U^\epsilon_h$ be the corresponding solution of
\eqref{eq:the scheme}--\eqref{int-cond}. By (a) again $|U_h-U_h^\epsilon|\leq 
\|u_0^\epsilon-u_0\|_0\leq C\epsilon$, and the estimate \eqref{K-eps} for
$K(u_0^{\epsilon})$ still holds.
Hence by part (e) we have that $|U_h^\epsilon(t_n,x)-u_0^\epsilon(x)|\leq 
CK(u_0^{\epsilon})t_n$, and then by the triangle inequality
$$|U_h(t_n,x)-u_0(x)|\leq C(\epsilon+K(u_0^{\epsilon})t_n+\epsilon).$$
In view of \eqref{K-eps}, we conclude by taking $\epsilon=0$,
$\epsilon=t_n$, $\epsilon=t_n^{\frac1\sigma}$ when $\sigma<1$, $\sigma=1$, $\sigma>1$  respectively.
 \end{proof}

Convergence of $U_h$ to the unique viscosity solution of
   \eqref{eqn:main}-\eqref{int-cond} follows from (an easy nonlocal
   extension of) the Barles-Perthame-Souganidis half-relaxed limits method
   \cite{Barles:1991tc} in view of monotonicity, stability,
   consistency of the scheme and strong comparison of the limit
   equation.

We now give precise estimates on the rate of convergence of our
method for our low-regularity solutions. These are the main
contributions of the paper. They are the first 
such result for non-convex degenerate equations of order greater than
one, the first result for nonlocal non-convex
equations, and these
estimates are more accurate than previous results for the non-local
operators $\I$: First, as expected, the rates depend on the maximal
fractional order of the operator $\I$, or 
equivalently, on $\sigma$ in assumptions \ref{A5}. But we also see a
surprising phenomenon that does not seem to have been observed before:
We have 3 different results depending on whether $\eta$ depends
on $(x,t)$, only on $x$, and on none of them. We devote one theorem to
each case:

\begin{thm}[General case]\label{main:thm}
Assume \ref{A1}--\ref{A5}, $0<\Delta x \leq\delta\leq1$,
the CFL condition \eqref{eq:CFL-condtion} holds, $u$ solves the equation
\eqref{eqn:main}-\eqref{int-cond}, and $U^{\delta}_h$ solves the scheme \eqref{eq:the
  scheme}-\eqref{int-cond}. Then there exists a constant $C>0$ (only
depending on the constants in \ref{A1}--\ref{A5}) such 
that for all $(t,x)\in \mathcal{G}_h^N$,
\begin{align*}
 \big|U^{\delta}_h(t,x) - u(t,x) \big| \leq C (1+T)
\begin{cases}(T\wedge1)^{\frac12}\bigg(\Delta t^{\frac12} + \Delta x ^{\frac12}+\delta^{1-\frac\sigma2}\bigg)&
  \text{if}\quad\sigma\in[0,1),\\[0.4cm] 
(T\wedge1)^{\frac1{2}}\bigg(\Delta
 t^{\frac1{2}}|\log \Delta t|+\Delta
 t^{\frac12}|\log \delta|+\Delta
 x^{\frac12}|\log \delta| +\delta^{\frac{1}{2}}\bigg)
&  \text{if}\quad \sigma=1,\\[0.4cm]
(T\wedge1)^{\frac1{2\sigma}}\Delta
 t^{\frac1{2\sigma}}+(T\wedge1)^{\frac12}\bigg(\Delta
 t^{\frac12}\delta^{1-\sigma}+\Delta
 x^{\frac12}\delta^{1-\sigma}
 +\delta^{1-\frac{\sigma}{2}}\bigg)
&  \text{if}\quad \sigma\in(1,2).
\end{cases} 
\end{align*}
\end{thm}

\begin{rem}
(a) These results imply the convergence of the scheme, and {\it optimal}
error estimates for local first order Hamilton-Jacobi equations
(cf. \cite{CL:2approx,souganidis1985}) follows as a special case since
$\mathcal I^{\alpha,\beta}\equiv 0$ is allowed. This
also means that the rate in the case $\sigma\in(0,1)$ is optimal
because of the first order drift term in our equation.
\medskip

\noindent(b) The results for $\sigma\in[1,2)$ are also optimal. The
principal error term is  $\delta^{1-\frac\sigma2}$ since $\Delta
x\leq \delta$. This term comes from the truncation of the
singularity and is optimal in view of the low regularity of our
problem.  See \eqref{eq:error-trunc} for the rate for smooth solutions
and Lemma \ref{lem_delta} below for the rate under our assumptions.


\end{rem}

\begin{thm}[No $t$-dependence]\label{main:thm2}
Let the assumptions of Theorem \ref{main:thm} hold and
 $\eta^{\A, \B}$ be independent of $t$.
\medskip

\noindent (a) Then there is a constant $C$ such that for all $(t,x)\in
\mathcal{G}_h^N$, 

 \begin{align*}
   |U^\delta_h - u| \le  C(1+T)
 \begin{cases}(T\wedge1)^{\frac12}\bigg(\Delta t^{\frac12} + \Delta x ^{\frac12}+\delta^{1-\frac\sigma2}\bigg)&
  \text{if}\quad\sigma\in[0,1),\\[0.4cm] 
(T\wedge1)^{\frac1{2}}\bigg(\Delta
 t^{\frac1{2}}|\log \Delta t|+\Delta
 x^{\frac12}|\log \delta| +\delta^{\frac{1}{2}}\bigg)
&  \text{if}\quad \sigma=1,\\[0.4cm]
  (T\wedge1)^{\frac1{2\sigma}}\Delta t^{\frac1{2\sigma}} +
   (T\wedge1)^{\frac12}\Big(\Delta x^{\frac12} \delta^{1-\sigma} +
   \delta^{1-\frac{\sigma}{2}}\Big) 
&  \text{if}\quad \sigma\in(1,2).
\end{cases} 
 \end{align*}

\noindent(b) If $K(u_0)<\infty$ (cf. Theorem \ref{wellpos_1}), then
there is a constant $C$ such that for all $(x,t)\in \mathcal{G}_h^N$,
 \begin{align*}
  |U^\delta_h - u| \le  C(1+T) (T\wedge1)^{\frac12}\bigg(\Delta t^{\frac12} +
   \Delta x^{\frac12} \Gamma(\sigma, \delta) + \delta^{1-\frac{\sigma}{2}} \bigg)\qquad
   \text{if}\qquad \sigma\in[0,2). 
\end{align*}
All the constants $C$ only depend on the constants in \ref{A1}--\ref{A5}
and \eqref{eq:alternate-cfl}, and for (b), also on $K(u_0)$.   
\end{thm}

The Theorem also holds for $\eta$ depending on $t$ if $\Delta t/\Delta x \leq K$.  
\begin{rem}
\noindent(a) Since $\eta$ depends on time, the convergence in $\Delta t$
and $\delta$ is coupled  in Theorem \ref{main:thm} for $\sigma\in[1,2)$! 
When $\eta$ does not depend on $t$, there is no coupling and a better rate
 by Theorem \ref{main:thm2} (a).\!\! 
\smallskip

\noindent (b) When $\sigma\geq1$, there is a
reduction of rate in $\Delta t$ because the solution of \eqref{eqn:main} no
longer is Lipschitz in $t$. However, assuming more regularity of the
initial data will make the solution $t$-Lipschitz again, and then we
get back the full rate $\frac12$ in Theorem \ref{main:thm2} (b).
\end{rem}

\begin{thm}[No $x,t$ dependence]\label{main:thm3}
The assumptions of Theorem \ref{main:thm} hold and $\eta^{\A, \B}$ is
independent of $x,t$. 
\medskip

\noindent(a) There is a constant $C$  such 
that for all $(t,x)\in \mathcal{G}_h^N$,
\begin{align*}
 \big|U^{\delta}_h(t,x) - u(t,x) \big| \leq C (1+T)
\begin{cases}
(T\wedge1)^{\frac1{2}}\bigg(\Delta
 t^{\frac1{2}}+\Delta
 x^{\frac12} +\delta^{1-\frac{\sigma}{2}}\bigg)
&  \text{if}\quad \sigma=[0,1),\\[0.4cm]
(T\wedge1)^{\frac1{2}}\bigg(\Delta
 t^{\frac1{2}}|\log \Delta t|+\Delta
 x^{\frac12} |\log \delta|^{\frac{1}{2}} +\delta^{\frac{1}{2}}\bigg)
&  \text{if}\quad \sigma=1,\\[0.4cm]
(T\wedge1)^{\frac1{2\sigma}}\Delta
 t^{\frac1{2\sigma}}+(T\wedge1)^{\frac12}\bigg(\Gamma(\sigma,\delta)^{\frac12}\Delta
 x^{\frac{1}{2}}
 +\delta^{1-\frac{\sigma}{2}}\bigg)
&  \text{if}\quad \sigma\in(1,2).
\end{cases} 
 \end{align*}

\noindent(b) If also $K(u_0)<\infty$ (cf. Theorem \ref{wellpos_1}), then
there is a constant $C$ such that for all $(t,x)\in \mathcal{G}_h^N$,
\begin{align*}
 \big|U^{\delta}_h(t,x) - u(t,x) \big| \leq C (1+T)(T\wedge1)^{\frac1{2}}\bigg(\Delta t^{\frac12} +
   \Gamma(\sigma,\delta)^{\frac12}\Delta x^{\frac12}+ \delta^{1-\frac{\sigma}{2}} \bigg)
 \end{align*}
All the constants $C$ only depend on the constants in \ref{A1}--\ref{A5}
and \eqref{eq:alternate-cfl}, and for (b), also on $K(u_0)$.   
\end{thm}

The proofs these results are given in Section \ref{proof}.

\begin{rem}\label{optd-rem}
(a) Our estimates hold for solutions that are merely Lipschitz in $x$ and
Lipschitz or H\"older in $t$. In general this is the best regularity
for our problem under our assumptions.
For more regular solutions, better estimates should hold. However, the
maximal rate or accuracy of our scheme is $O(\Delta x^{1 \wedge (2-\sigma)}).$
The dominant error term comes from truncation of the measure
(cf. \eqref{eq:error-trunc}, 
\eqref{eq:error-drift}, \eqref{eq:nonlocal-error} and recall that
$\Delta x\leq\delta$). 
\medskip

\noindent(b)  The choices of $\delta$ that optimize the
error are $\delta=\Delta x \ \text{for}\ \sigma\in(0,1)$ and when $\sigma\in(1,2)$ that are $\delta=\max(\Delta t^{\frac1\sigma},\Delta x^{\frac1\sigma})$ in Theorem \ref{main:thm},  $\delta= \Delta x^{\frac{1}{\sigma}}$ in Theorem \ref{main:thm2} and $\delta= \Delta x$ in Theorem \ref{main:thm3}.   
Assume now $K(u_0)=\infty$. When $\sigma\leq1$, all cases
then lead to the estimate
\begin{align*}
 |u- U_h^\delta| \leq C
\left\{  
		\begin{array}{ll}
		 (\Delta t)^{\frac{1}{2}} + (\Delta x)^{\frac{1}{2}} & \quad \mbox{when} \quad \sigma \in[0,1),\\[0.2cm]
		 (\Delta t)^{\frac{1}{2}}| \log ( \Delta t)| + (\Delta x)^{\frac{1}{2}} | \log ( \Delta x)| & \quad \mbox{when} \quad \sigma =1.
		\end{array}
\right.
\end{align*}
However the estimates for $\sigma\in(1,2)$ are different in each case:
 \begin{align*}
 |u- U_h^\delta| \leq C
\left\{  
		\begin{array}{ll}
		 (\Delta t)^{\frac{2-\sigma}{2\sigma}} + ( \Delta
                  x)^{\frac{2-\sigma}{2\sigma}} & \quad \mbox{in
                    Theorem \ref{main:thm}}, \\[0.2cm]
		 (\Delta t)^{\frac{1}{2\sigma}} + ( \Delta x)^{\frac{2-\sigma}{2\sigma}} & \quad \mbox{in
                    Theorem \ref{main:thm2} (a)}, \\[0.2cm]
		 (\Delta t)^{\frac{1}{2\sigma}} + ( \Delta x)^{\frac{2-\sigma}{2}} & \quad  \mbox{in
                    Theorem \ref{main:thm3} (a)}.
		\end{array}
\right.
 \end{align*}
Note the improvement in rate in each line! 
When $K(u_0)<\infty$, the solution is Lipschitz in $t$, and the time
rate improves to $O(\Delta t^{\frac12})$ in Theorems \ref{main:thm2}
and \ref{main:thm3}. In particular, the rate of Theorem
\ref{main:thm3} (b) becomes
$$O\Big((\Delta t)^{\frac12}+( \Delta
                  x)^{\frac{2-\sigma}{2}}\Big).$$
This latter spatial rate is
consistent with the rates (for the implicit scheme) of Theorem 6.3 in \cite{CJ2011} where other
types of $(x,t)$-independent nonlocal nonlinear equations are considered. 
\end{rem}

 \section{Proof  of the main results - Theorems \ref{main:thm}, \ref{main:thm2} and
   \ref{main:thm3}} \label{proof}

\subsection{Reduction to finite L\'{e}vy measures}
\label{sec:red}

Since the two problems \eqref{eqn:main} and
\eqref{eqn:main-perturbed-1}  have the same data and coefficients
except for the L\'{e}vy measures $\nu$ and $\nu_\delta$, we can use the continuous
dependence results of \cite{Jakobsen:2005jy} to bound  the
distance between $u$ and $u^{\delta}$.
\begin{lem}\label{lem_delta}
Assume \ref{A1}--\ref{A5}. If $u$ and $u^{\delta}$ solve
\eqref{eqn:main} and \eqref{eqn:main-perturbed-1}, then 
 \begin{align*}
|u(t,x)- u^{\delta}(t,x)| \leq  C
                            T^{\frac12}\delta^{1-
                            \frac{\sigma}{2}}\qquad\text{for all}\qquad (t,x)\in Q_T.
\end{align*}
\end{lem}
\begin{proof}
In a similar way as Theorem 4.1 in \cite{Jakobsen:2005jy} follows from Corollary
3.2 in \cite{Jakobsen:2005jy}, we can use Corollary
3.2 in \cite{Jakobsen:2005jy} and the fact that all coefficients are
bounded to show that
$$|u(t,x)- u^{\delta}(t,x)| \leq C T^{\frac12} \sqrt{\int_{\R^M \setminus
                            \{0\}} |\Eta(t,x;z)|^2 \ |\nu -
                            \nu_{\delta}|(dz)}.$$
Note that as opposed to Theorem 4.1 in \cite{Jakobsen:2005jy}, there is
no growth in $x$ in our estimate. The result then follows by
\ref{A3}--\ref{A5} and $\int_{|z|<\delta}|z|^2\nu(dz)= C\delta^{2-\sigma}$.
\end{proof} 
In view of the result, it is enough to prove Theorem \ref{main:thm}
when the L\'evy measure $\nu$ is replaced by the bounded measure
$\nu_\delta$. Therefore, in the rest of the proof we only work with
$u=u^\delta$, the solution of
\eqref{eqn:main-perturbed-1}--\eqref{int-cond}.

\subsection{The doubling of variables argument}
\label{sec:dbl}

 Recall that $U_h$ is defined on $\mathcal{G}_h^N$ as
 $U_h(t_n, x_{\bj})= U_{\bj}^n$, and $u=u^\delta$ solves \eqref{eqn:main-perturbed-1}--\eqref{int-cond}. 
We want to  bound $|U_h(t_n, x_{\bj})- u(t_n, x_{\bj})|$ in $\mathcal{G}_h^N$ and start by deriving a nonnegative upper
 bound on
\begin{align*}
\mu = \sup_{\bj \in \zn, \, n\leq M} (U^n_{\bj} - u(t_n, x_{\bj})).
\end{align*}
Assume that $\mu > 0$ (if not $\mu\leq 0$ and we are done).
Since $u$ and $U_h$ are bounded uniformly in $h$,
$$R:=\max\{\|u\|_{L^\infty}, \|U_h\|_{L^\infty}\}<\infty.$$ 
We will use the method of doubling of variables (e.g
\cite{CL:2approx}) and to do that we introduce $\Psi
:\mathcal{G}_h^N\times  Q_T  \rightarrow \R$, 
\begin{align*} 
\Psi(t,x,s,y) = U_h(t,x)- u(s,y) -
  \phi(x,y) - \xi(t,s)  -\frac{\mu}{4T} (t+s), 
\end{align*}
where $\phi : \rn\times \rn \rightarrow \R$
and $\xi : [0,T]\times [0,T] \rightarrow \R$ are defined by
\begin{align*}
 \phi(x,y) =  \frac{\gamma}{2} |x-y|^2 + \frac{\varepsilon}{2}(|x|^2 + |y|^2) \qquad\text{and}\qquad
 \xi(t,s)  =  \frac{\eta}{2}|t-s|^2,
 \end{align*}
 for $\gamma ,\eta, \varepsilon>0$. From the boundedness of $U_h$ and $u$, it follows that $\Psi$ has a
 maximum at $(t_0,x_0,s_0,y_0)\in\mathcal{G}_h^N\times Q_T$ such that
 \begin{align}
 \label{max_pt}
 \Psi(t_0,x_0,s_0,y_0) \geq \Psi(t,x,s,y)
 \end{align}
 for any $(t,x,s,y) \in \mathcal{G}_h^N\times Q_T$. Since  $0=
 \Psi(0,0,0,0)\leq \Psi(t_0,x_0,s_0,y_0)$, it follows that
\begin{align*}
 \frac{\gamma}2 |x_0-y_0|^2 + \frac\varepsilon2 (|x_0|^2 +
 |y_0|^2) + \frac{\eta}2|t_0-s_0|^2\leq U_h(t_0,x_0) - u(s_0,y_0),
\end{align*}
and hence  $U_h(t_0, x_0) - u(s_0,y_0) \geq0$ and
\begin{align}
  \label{eq:doublinboundsg-}
 \varepsilon (|x_0|^2 + |y_0|^2) \leq 4R.
 \end{align}
Moreover, since the map
$y \rightarrow u(s_0,y) + \phi(x_0,y) $
has a minimum at $y_0$ and $u$ is Lipschitz, 
$\phi(x_0, y_0) - \phi(x_0, y)  \leq u(s_0, y ) - u(s_0,y_0)\leq L
                                |y-y_0|$,
and hence $|D_y \phi(x_0,y_0)| \leq L$. Then by the definition of
$\phi$, and since $\varepsilon|y_0|\leq \sqrt{4R\varepsilon}$
by the last bound in \eqref{eq:doublinboundsg-}, we have
\begin{align}\label{esti_1}
|x_0 - y_0| \leq \frac{1}{\gamma}(L+\sqrt{4R\varepsilon}).
\end{align} 
By the inequality $\Psi(t_0,x_0,t_0,y_0)\leq  \Psi(t_0,x_0,s_0,y_0)$
and the regularity of $u$ in Theorem \ref{wellpos_1},
we find that
 \begin{align}\label{esti_00}
  \frac{\eta}2|t_0-s_0|^2\leq u(t_0,y_0) - u(s_0,y_0)\leq K\omega(t_0-s_0)
 \end{align}
where $\omega(r)=|r|$ if $K(u_0)<\infty$ and $\omega=\bar\omega$ from
Theorem \ref{wellpos_1} (c) if not. For $\sigma\neq1$, we get that
\begin{align}
\label{esti_2} |t_0 - s_0| & \leq \frac{2K}{\eta^{q}}
\end{align}
with $q=1$ when $K(u_0)<\infty$ and $u$ is Lipschitz in $t$ and
$q=\frac{\sigma}{2\sigma-1}$ when $\sigma\in(1,2)$ and
$u$ is H\"older $\frac1\sigma$ in $t$.

If either $t_0$ or $s_0$ is $0$, then we get a bound on $\mu$ using
only the regularity of the $u$ and $U_h$ at $t=0$. If $s_0 = 0$ and
$t_0>0$, then for any point $(t,x)\in \mathcal{G}_h^N$,  
\begin{align*}
&U_h(t,x)  - u(t,x) - \varepsilon|x|^2 -\frac{\mu}{2T}t=\Psi (t,x,t,x)\leq \Psi (t_0,x_0,0,y_0)\\
& = U_h(t_0, x_0)- u_0(y_0) - \phi(x_0,y_0) - \xi(t_0,0)
-\frac{\mu}{4T} t_0\leq  U_h(t_0, x_0)- u_0(y_0).
\end{align*} 
If either \ref{A5} holds with $\sigma\in(0,1)$ or $K(u_0)<\infty$,
then $u$ and $U_h$ are Lipschitz in $t$ at $t=0$. By Theorem
\ref{thm:existence-scheme} (e) and the regularity of $u_0$, 
$U_h(t_0, x_0) -u_0(x_0) + u_0(x_0)- u_0(y_0) 
\leq C ( t_0+ |x_0 - y_0|)$.
Hence by estimates  \eqref{esti_1} and \eqref{esti_2} with $q=1$ and since
$t_0=|t_0-s_0|$, we find that
$U_h(t,x) - u(t,x) -\varepsilon |x|^2 - \frac{\mu}{2T}t \leq
C(\frac{1}{\gamma}+\frac1{\eta})$. If we first send $\varepsilon$ to $0$
and then take the supremum over $\mathcal{G}_h^N$, by
  the definition of $\mu$ we get that 
\begin{align*}   
\frac{\mu}{2} \le  \sup_{\bj \in \Z^N , n\leq
  M} (U^n_{\bj} - u(t_n, x_\bj)) -\frac{\mu}{2} \leq
  C\Big(\frac{1}{\gamma}+\frac1{\eta}\Big).  
 \end{align*}
When $\sigma \in (1,2)$,  $u$ and $U_h$ are only H\"older $\frac1\sigma$ in $t$ at $t=0$
(cf. Theorem \ref{wellpos_1} (c) and Theorem \ref{thm:existence-scheme} (d)). In this case e.g.
$U_h(t_0,x_0)-u_0(y_0)\leq C(t_0^{\frac1\sigma}+|x_0-y_0|)$, and hence by
\eqref{esti_1} and \eqref{esti_2} with $q=\frac\sigma{2\sigma-1}$, we find that
\begin{align*} 
\frac{\mu}{2} \leq   C\Big(\frac{1}{\gamma}+\frac1{\eta^{\frac1{2\sigma-1}}}\Big).
\end{align*}
A similar argument using time regularity of $u$, shows that
these bounds also hold when $t_0=0$ and $s_0\geq0$.

Only the case $t_0>0$ and $s_0 >0$ remains. Here we have to use the
equations and the argument is long so we divide it into several steps. 

\medskip
\noindent\textbf{Step 1:} It is easily seen from \eqref{max_pt} that
$(s_0,y_0)$ is a global minimum point on $Q_T$ of   
$$(s,y) \rightarrow u(s,y) -\Big(-\phi(x_0,y) - \xi(t_0,s) - \frac{\mu}{4T}(t_0+s)\Big). $$
By the supersolution inequalities for $u$
(cf. \eqref{eqn:main-perturbed-1}) with
  $\kappa=\delta$,
\begin{align} \label{sup_soln}
 - D_s \xi(t_0, s_0) - \frac{\mu}{4T}+\inf_{\A}\sup_{\B} \bigg\{&
                                                                     -f^{\A,\B}(s_0,y_0) +c^{\A,\B}(s_0,y_0) u(s_0,y_0)\notag \\
& - \tilde b^{\A,\B}_\delta(s_0,y_0)  (-D_y
                                                                     \phi(x_0,y_0))   -\mathcal{J}^{\A,\B,\delta}[u](s_0,y_0) \bigg\} \geq 0 . 
\end{align}
  We now get an analogous relation for the scheme at the
  grid-point $(t_0,x_0)$. By \eqref{max_pt} again
  $\Psi(t,x,s_0,y_0)\leq \Psi(t_0,x_0,s_0,y_0)$, and hence the function
\begin{align*}
W(t,x):=U_h(t_0,x_0)+ \phi(x,y_0) - \phi(x_0,y_0) +\xi(t,s_0 )- \xi(t_0,s_0 ) + \frac{\mu}{4T}(t-t_0)
\end{align*}
satisfies 
\begin{align*}
U_h\leq W\quad\text{in}\quad \mathcal{G}^N_h \qquad \text{and}\qquad U_h(t_0,x_0)= W(t_0,x_0).
\end{align*}
By the definition and monotonicity of the scheme (under the CFL
condition \eqref{eq:CFL-condtion}) we then get at the maximum
point $(t_0, x_0)= (p\Delta t, \Delta x\, \bk)$ that
\begin{align*}
\notag U^{p}_{\bk} &= U^{p-1}_{\bk} -\Delta t \inf_{\A}\sup_{\B} \left\{- f^{\A,\B,p}_{\bk}+c^{\A,\B,p}_{\bk} U^{p-1}_{\bk}  -\theta \mathcal{D}_h^{\A,\B,\delta}[U]^{p}_{\bk} - (1-\theta)\mathcal{D}_h^{\A,\B,\delta}[U]^{p-1}_{\bk} \right. \notag \\
& \hspace*{6cm} \left.   - \vartheta \mathcal{J}^{\A,\B,\delta}_h[U]^{p}_{\bk}
  -(1-\vartheta)\mathcal{J}^{\A,\B,\delta}_h[U]^{p-1}_{\bk}
\right\}\notag\\
&\leq  W^{p-1}_{\bk} -\Delta t \inf_{\A}\sup_{\B} \left\{- f^{\A,\B,p}_{\bk}+c^{\A,\B,p}_{\bk} W^{p-1}_{\bk}  -\theta \mathcal{D}_h^{\A,\B,\delta}[W]^{p}_{\bk} - (1-\theta)\mathcal{D}_h^{\A,\B,\delta}[W]^{p-1}_{\bk} \right. \notag \\
& \hspace*{6cm}\left.
   - \vartheta \mathcal{J}^{\A,\B,\delta}_h[U]^{p}_{\bk}
  -(1-\vartheta)\J_h^{\A,\B,\delta}[U,W]^{p-1}_{\bk}
\right\}\notag
\end{align*}
where
\begin{align*}
\J_h^{\A,\B,\delta}[U,W]^{p-1}_{\bk} =  \sum_{\bj \in \zn} \left[ U^{p-1}_{\bk+ \bj}-W^{p-1}_{\bk} \right]\int_{|z| \geq \delta} \omega_{\bj}\left( \Eta(t_0-\Delta t, x_0; z);h\right) \ \nu(dz)
\end{align*}
and this non-standard term is admissible by the monotonicity of the
full scheme in the $U^{p-1}_\bk$-argument. We will see later that we
really need the term in this form.
By definition of $W$,
$\mathcal{D}^{\A,\B}_h[W]=\mathcal{D}^{\A,\B}_h[\phi(.,y_0)]$ etc.,
and we
divide by $\Delta t$ and rewrite the above inequality as
\begin{align}\label{esti_4}
\frac{\mu}{4T}  \leq &\ \frac{\xi(t_0-\Delta t,s_0)-
  \xi(t_0,s_0)}{\Delta t} -   \inf_{\A} \sup_{\B} \bigg\{ -f^{\A,\B,p}_{\bk} +c^{\A,\B,p}_{\bk} \Big[ U^p_{\bk} + \xi(t_0-\Delta t, s_0)  -\xi(t_0, s_0)- \frac{\mu}{4T}\Delta t \Big]  \notag\\
&\hspace*{0.5cm}\left. - \theta \mathcal{D}^{\A,\B,\delta}_h[\phi(.,y_0)](t_0,x_0)- (1-\theta) \mathcal{D}^{\A,\B,\delta}_h[\phi(.,y_0)](t_0-\Delta t,x_0)\right. \notag\\
& \hspace*{0.5cm}- (1-\vartheta) \J_h^{\A,\B,\delta}[U_h,W](t_0-\Delta t,x_0)- \vartheta
\mathcal{J}_h^{\A,\B,\delta}[U_h](t_0,x_0)  \bigg\}.
 \end{align}
Subtracting inequalities \eqref{sup_soln} and \eqref{esti_4} and
using the fact that $\inf\sup f- \inf\sup g\leq \sup\sup (f-g)$,
\begin{align}\label{diff}
\frac{\mu}{2T}  \leq &\ \frac{\xi(t_0-\Delta t,s_0)-
  \xi(t_0,s_0)}{\Delta t}- D_s \xi(t_0, s_0)\notag\\
& +\sup_{\alpha}\sup_{\beta}\bigg\{f^{\A,\B,p}_{\bk}
-f^{\A,\B}(s_0,y_0)- c^{\A,\B,p}_{\bk} \Big[ U^p_{\bk} +
\xi(t_0-\Delta t, s_0) - \frac{\mu}{4T}\Delta t  -\xi(t_0,
s_0)\Big]+c^{\A,\B}(s_0,y_0) u(s_0,y_0) \notag\\
&\qquad\qquad\quad +\theta \mathcal{D}^{\A,\B,\delta}_h[\phi(.,y_0)](t_0,x_0)+ (1-\theta)
\mathcal{D}^{\A,\B,\delta}_h[\phi(.,y_0)](t_0-\Delta t,x_0) -
\tilde b^{\A,\B}_\delta(s_0,y_0)  (-D_y \phi(x_0,y_0))\notag\\
&\qquad\qquad\quad +(1-\vartheta) \J_h^{\A,\B,\delta}[U_h,W](t_0-\Delta t,x_0) + \vartheta
\mathcal{J}_h^{\A,\B,\delta}[U_h](t_0,x_0) -\mathcal{J}^{\A,\B,\delta}[u](s_0,y_0) \bigg\}\notag\\
=&\ I_1 + \sup_{\alpha}\sup_{\beta}\bigg\{I_2+ I_3 +I_4\bigg\}.
\end{align}
\smallskip

\noindent\textbf{Step 2:}
We now estimate the terms $I_1,I_2,I_3$ in
\eqref{diff}. First note that    
$\xi(t_0-\Delta t,s_0 ) - \xi(t_0,s_0 )=- \partial_t \xi(t_0,s_0)\Delta t  +  \frac \eta2 \Delta t^2$,
and hence since $\partial_t \xi=-\partial_s \xi$ 
$$I_1=\frac{\xi(t_0-\Delta t,s_0)-
  \xi(t_0,s_0)}{\Delta t}- \partial_s \xi(t_0, s_0)= \frac \eta2 \Delta t.$$
We estimate $I_2$ using $c\geq 0$, $U^p_{\bk}-u(s_0,y_0)\geq0$,
regularity of the coefficients $c$ and $f$, the estimate on $I_1$, and the
bounds on $|x_0-y_0|$ and $|t_0-s_0|$,
\begin{align}
\notag I_2&=-c^{\A,\B,p}_{\bk} \Big[ U^p_{\bk} +
\xi(t_0-\Delta t, s_0) - \frac{\mu}{4T}\Delta t  -\xi(t_0,
s_0)\Big]+c^{\A,\B}(s_0,y_0) u(s_0,y_0) + f^{\A,\B}(t_0,x_0) - f^{\A,\B}(s_0,y_0) \\
\notag &\leq 0 + | u(s_0,y_0)||c^{\A,\B,p}_{\bk}-c^{\A,\B}(s_0,y_0)|+K \Big(|\xi(t_0-\Delta t, s_0)   -\xi(t_0,
s_0)|+ \frac{\mu}{4T}\Delta t\Big) + \big|  f^{\A,\B}(t_0,x_0) - f^{\A,\B}(s_0,y_0)\big|\\
&\leq C\Big(|x_0-y_0|+|t_0-s_0|+ \Delta t +\eta\Delta t^2\Big).\label{eq:bound-on-I2}
\end{align}
We now estimate $I_3$. By the consistency estimate \eqref{eq:error-drift}, the definition
of $\tilde b^{\A,\B}_\delta$, the time regularity and bounds on $b$ and $\eta$ and
integrability assumptions \ref{A2}--\ref{A5} of $\nu$, the definition and gradient
bound of $\phi$,  
\begin{align*}
& \theta \mathcal{D}^{\A,\B,\delta}_h[\phi(.,y_0)](t_0,x_0)+ (1-\theta)
\mathcal{D}^{\A,\B,\delta}_h[\phi(.,y_0)](t_0-\Delta t,x_0)\\
&\leq \left(\theta \tilde b^{\A,\B}_\delta(t_0,x_0)+ (1-\theta)\tilde b^{\A,\B}_\delta(t_0-\Delta t,x_0)
\right)\cdot D_x\phi(x_0,y_0) + C\|\tilde
  b^{\A,\B}_\delta\|_0\|D^2\phi\|_0\Delta x \\ 
&\leq   \tilde b^{\A,\B}_\delta(t_0,x_0) \cdot
  D_x\phi(x_0,y_0)+C\Big(1+\int_{|z|>\delta}\rho(z)\nu(dz)\Big)\Big((1-\theta)L\Delta t+(\gamma+\varepsilon)\Delta x
\Big).
\end{align*}
Hence since $D_x\phi=-D_y\phi +\varepsilon (x+y)$ and $b$ is Lipschitz continuous,
by the maximum point estimates, the definition
of $\tilde b^{\A,\B}_\delta$,
and the Lipschitz bound on $\phi$,
\begin{align}\notag
I_3&=\theta \mathcal{D}^{\A,\B,\delta}_h[\phi(.,y_0)](t_0,x_0)+ (1-\theta)
\mathcal{D}^{\A,\B,\delta}_h[\phi(.,y_0)](t_0-\Delta t,x_0) -
\tilde b^{\A,\B}_\delta(s_0,y_0) \cdot (-D_y \phi(x_0,y_0))\\
\notag &\leq \Big(\tilde b^{\A,\B}_\delta(t_0,x_0)-\tilde b^{\A,\B}_\delta(s_0,y_0) \Big)\cdot
D_x\phi(x_0,y_0) +\varepsilon
         |x_0+y_0||\tilde b^{\A,\B}_\delta(s_0,y_0)|
\\
&\notag\quad 
+C\Big(1+\int_{|z|>\delta}\rho(z)\nu(dz)\Big)\Big((1-\theta)\Delta
  t+(\gamma+\varepsilon)\Delta x\Big) \\
\label{eq:gen1}&\leq C\Big(1+\int_{|z|>\delta}\rho(z)\nu(dz)\Big)\Big((|x_0-y_0|+|t_0-s_0|)L+(1-\theta)\Delta t+(\gamma+\varepsilon)\Delta  
x  \Big)+ o_\varepsilon(1).
\end{align}
In the case that $\eta$ does not depend on $t$, then a recomputation
of the above estimate using the fact that $\tilde
b^{\A,\B}_\delta(x,t):=b^{\A,\B}(x,t)-\int_{|z|>\delta}
\eta^{\A,\B}(x;z) \, \nu(dz) $,  leads to  
\begin{align}\label{eq:indepd-t-1}
I_3&\leq C\Big(|t_0-s_0|+(1-\theta)\Delta t+\Big(1+\int_{|z|>\delta}\rho(z)\nu(dz)\Big)\Big(|x_0-y_0|+(\gamma+\varepsilon)\Delta  
x\Big)  \Big)+o_\varepsilon(1).
\end{align}

When $\eta$ does not depend on both $x$ and $t$ then
\begin{align} \label{eq:indepd-t-x-1}
I_3&\leq C\Big(|t_0-s_0|+|x_0-y_0|+(1-\theta)\Delta t+\Big(1+\int_{|z|>\delta}\rho(z)\nu(dz)\Big)(\gamma+\varepsilon)\Delta  
x  \Big)+o_\varepsilon(1).
\end{align}

\smallskip

\noindent\textbf{Step 3:}
It remains to estimate $I_4$. We rewrite this term as
\begin{align*}
I_4=\ & \vartheta \left[\mathcal
        J^{\A,\B,\delta}_h[U_h](t_0,x_0)-\mathcal{J}^{\A,\B,\delta}[u](s_0,y_0)
        \right]  +(1-\vartheta) \left[\tilde{\mathcal J}_h^{\A,\B,\delta}[U_h,W](t_0-\Delta t,x_0) -\mathcal{J}^{\A,\B,\delta}[u](s_0,y_0) \right] \\
\equiv\ &  \vartheta  J_1  + (1-\vartheta) J_2.
\end{align*} 
By the definition of $W$ and since $\sum \omega_{\bj}(x;h) = 1$, we find that
\begin{align*}
 J_2 = & \int_{|z|>\delta} \sum_{\bj \in \zn} \bigg\{  u(s_0,y_0) - u(s_0, y_0 + \Eta(s_0,y_0;z)) \notag \bigg. \\
& \hspace*{2cm} \left. - \Big(U^p_{\bk} - \xi(t_0,s_0)+ \xi(t_0-\Delta t,s_0) - \frac{\mu}{4T}\Delta t\Big)  +  U^{p-1}_{\bk+ \bj}\right\} \omega_{\bj}(\Eta(t_0-\Delta t,x_0;z);h) \ \nu(dz). 
\end{align*}
In the following argument, it is essential that we
  have $U^p_{\bk}$ in the integral defining $J_2$ and not
  $U^{p-1}_{\bk}$, and this explains why we introduced the strange
  quantity $\tilde{\mathcal J}_h^{\A,\B,\delta}[U_h,W]$ in the first
  place.
 Recall that $(t_0,x_0,s_0,y_0)$ is a global maximum point of $\Psi$, so $\Psi(t_0,x_0,s_0, y_0) \geq \Psi (t_0-\Delta t, x_0 +
x_{\bj}, s_0, y_0 + \Eta(s_0,y_0;z))$, and hence
\begin{align*}
& u(s_0,y_0) - u(s_0, y_0 + \Eta(s_0,y_0;z)) - \Big(U^p_{\bk} - \xi(t_0,s_0)+ \xi(t_0-\Delta t,s_0) - \frac{\mu}{4T}\Delta t\Big)  +  U^{p-1}_{\bk+ \bj} \\
&\leq \phi(x_0 + x_{\bj} , y_0 + \Eta(s_0,y_0;z)) - \phi(x_0, y_0).
\end{align*} 
By the nonnegativity of $\omega_{\bj}$, the definition of the
interpolation $\gi_h$, the error bound \eqref{int_err}, and assumptions \ref{A3} and \ref{A4}, we may
use these inequalities to estimate $J_2$:
\begin{align}
   J_2 \leq & \int_{|z|>\delta} \sum_{\bj \in
    \zn} \left\{ \phi(x_0 + x_{\bj} , y_0 + \Eta(s_0,y_0;z)) -
    \phi(x_0, y_0)  \right\}\omega_{\bj}(\Eta(t_0-\Delta t,x_0;z);h) \ \nu(dz)  \notag \\
  = & \int_{|z|>\delta} \left\{ \gi_h[\phi(x_0 + \cdot
    , y_0 + \Eta(s_0,y_0;z))](\Eta(t_0-\Delta t,x_0;z)) -
    \phi(x_0,y_0)\right\} \ \nu(dz) \notag \\
  \leq & \int_{|z|>\delta} \left\{ \phi(x_0 +
    \Eta(t_0-\Delta t,x_0;z) , y_0 + \Eta(s_0,y_0;z)) - \phi(x_0,y_0)+
    K (\gamma+\varepsilon) \ (\Delta x)^2\right\} \ \nu(dz) \notag \\
  =  & \int_{|z|>\delta}  \bigg\{ \gamma (x_0-y_0)\cdot(\Eta(s_0,y_0;z)-\Eta(t_0-\Delta t,x_0;z))+  \frac{\gamma}{2}|\Eta(s_0,y_0;z)-\Eta(t_0-\Delta t,x_0;z)|^2 \notag \\
  &\hspace*{1cm}+ \varepsilon \Big(x_0\cdot \Eta(t_0-\Delta t,x_0;z) +
  y_0\cdot\Eta(s_0,y_0;z)\Big)   +
  \frac{\varepsilon}{2}(|\Eta(t_0-\Delta t,x_0;z)|^2 +
  |\Eta(s_0,y_0;z)|^2)  \notag \\
  & \hspace*{1cm} + K (\gamma+\varepsilon) \ (\Delta x)^2 \bigg\} \ \nu(dz)  \notag \\
  \leq &\ C \gamma \left\{|x_0-y_0|\Big(|x_0-y_0|+ |t_0-s_0|+ \Delta t\Big)
    \int_{|z|>\delta} \rho(z) \ \nu(dz)+\Big(|x_0-y_0|^2+ |t_0-s_0|^2 + \Delta t^2\Big)
   \int_{|z|>\delta} \rho(z)^2 \ \nu(dz)\right\} \notag\\
&  \qquad +C \varepsilon(|x_0|+ |y_0|) \int_{|z|>\delta} \rho(z) \ \nu(dz)+ C \varepsilon \int_{|z|>\delta} \rho(z)^2 \ \nu(dz)  + C (\gamma+\varepsilon)
  (\Delta x)^2\int_{|z|>\delta} \ \nu(dz).\label{eq:gen2}
\end{align}

In the case that $\eta$ does not depend on $t$, an easy recomputation of the
above estimate shows that
\begin{align}
\label{eq:indepd-t-2}J_2& \leq  C\Big( \gamma|x_0-y_0|^2+\varepsilon(1+|x_0|+ |y_0|)\Big)
    \int_{|z|>\delta} (\rho(z) + \rho(z)^2) \ \nu(dz) + C (\gamma+\varepsilon)
  (\Delta x)^2\int_{|z|>\delta} \ \nu(dz),
\end{align}

and when $\eta$ does not depend on both $x$ and $t$ then
\begin{align}
\label{eq:indepd-t-x-2}J_2& \leq  C\varepsilon(1+|x_0|+ |y_0|)
    \int_{|z|>\delta} (\rho(z) + \rho(z)^2) \ \nu(dz) + C (\gamma+\varepsilon)
  (\Delta x)^2\int_{|z|>\delta} \ \nu(dz). 
\end{align}

Similar but simpler arguments, using the  fact that
$\Psi(t_0,x_0,s_0,y_0) \geq \Psi(t_0,x_0 + x_{\bj}, s_0, y_0 +
\Eta(s_0,y_0;z))$, shows that $J_1$, and hence also $I_4$, satisfy the
same upper bounds as $J_2$.

\medskip

\noindent\textbf{Step 4:} By 
  \ref{A3}-\ref{A5} and the 
  definition of $\Gamma(\sigma,\delta)$,
  \begin{align*}
& \int_{|z|>\delta} \rho(z)^2 \ \nu(dz) \leq K^2 \int_{0<|z|<1} |z|^2 \
  \nu(dz) + \int_{|z|>1} \rho(z)^2 \ \nu(dz) \leq C, \\ 
 &  \int_{|z|>\delta} \rho(z) \ \nu(dz) \leq K^2 \int_{\delta<|z|<1} |z| \ \nu(dz)
   + \int_{|z|>1} \rho(z)^2 \ \nu(dz)\leq C(1+\Gamma(\sigma,\delta)), \\ 
& \int_{|z|>\delta} \ \nu(dz)\leq
C\int_{\delta<|z|<1}\frac{dz}{|z|^{M+\sigma}}+C\leq
  C(1+\delta^{-\sigma}). 
\end{align*}
Now we get a bound on $\mu$ from \eqref{diff} by using these
estimates along with the estimates of steps 1--3 (which are
independent of $\alpha$ and $\beta$). If we also take into
account the fact that $0<\Delta x<\delta\leq 1$,
$\Gamma(\sigma,\delta)\geq1$, and that we may take $\eta,\gamma\geq1$
and $\Delta t\leq 1$, we find after combining \eqref{eq:bound-on-I2}, \eqref{eq:gen1} $\&$\eqref{eq:gen2} and  dropping all non-dominant terms that
\begin{align} \label{sigma_bound}
\notag \frac\mu{2T} \leq & \ 
                 I_1+\sup_\alpha\sup_\beta\Big\{I_2+I_3+I_4\Big\}\\ 
\notag \leq & \ C\bigg\{\eta\Delta t+\gamma\Delta t^2+\gamma |t_0-s_0|^2+\gamma\frac{\Delta x^2}{\delta^{\sigma}}\bigg\}\\
\notag & + C\Gamma(\sigma,\delta) \bigg\{|x_0-y_0|+|t_0-s_0|+\Delta
  t+\gamma\Delta x+\gamma|x_0-y_0|\Big(|x_0-y_0|+
  |t_0-s_0|+ \Delta t\Big) \bigg\}\\ 
 & + C\varepsilon \bigg\{1+\Gamma(\sigma,\delta)\Big(|x_0|+ 
  |y_0|+\Delta x\Big) +  \frac{\Delta
  x^2}{\delta^{\sigma}}\bigg\}.
\end{align}
Note that by \eqref{eq:doublinboundsg-}, all $\varepsilon$-terms go to $0$ as
$\varepsilon\to0$ and $\gamma,\eta,\delta$ are fixed, and $\gamma\frac{\Delta 
  x^{2}}{\delta^\sigma}\leq\gamma\Delta x\,\delta^{1-\sigma}\leq
\Gamma(\sigma,\delta)(\gamma\Delta x)$ since $\Delta x\leq \delta$. Hence in
view of estimates \eqref{eq:doublinboundsg-}--\eqref{esti_2},
\begin{align*} 
  \frac\mu{2T} \leq C\bigg(\eta\Delta t+\gamma\Delta t^2+\frac\gamma{\eta^{2q}} +\Gamma(\sigma,\delta)\Big( \frac1{\gamma}+\frac1{\eta^q} +\Delta
 t+\gamma\Delta x\Big)   \bigg)+o_\varepsilon(1).
\end{align*}  
In the case that $\eta$ does not depend on $t$, we combine  \eqref{eq:bound-on-I2}, \eqref{eq:indepd-t-1} and \eqref{eq:indepd-t-2} and find 
\begin{align*}
\frac\mu{2T} \leq & \ C\bigg\{\eta\Delta
t+|t_0-s_0|+\gamma\frac{\Delta 
                       x^2}{\delta^{\sigma}}+\Gamma(\sigma,\delta)\Big(|x_0-y_0|+\gamma\Delta x+\gamma|x_0-y_0|^2\Big)\bigg\}+o_\varepsilon(1)\\
                       \leq & \ C\bigg\{\eta\Delta
                       t+\frac1{\eta^q}
                       +\Gamma(\sigma,\delta)\Big(\frac1\gamma+\gamma\Delta
                       x\Big)\bigg\}+o_\varepsilon(1),
\end{align*}
and when $\eta$ does not depend on both $x$ and $t$ then   \eqref{eq:bound-on-I2}, \eqref{eq:indepd-t-x-1} and \eqref{eq:indepd-t-x-2} are combined to have
\begin{align*}
\frac\mu{2T} \leq & \ C\bigg\{\eta\Delta
t+|t_0-s_0|+|x_0-y_0|+\gamma\frac{\Delta 
                       x^2}{\delta^{\sigma}}+\Gamma(\sigma,\delta)\gamma\Delta x\bigg\}+o_\varepsilon(1)\\ \leq &\ C\bigg\{\eta\Delta
                       t+\frac1{\eta^q} + \frac{1}{\gamma}
                       +\Gamma(\sigma,\delta)\gamma\Delta x\bigg\}+o_\varepsilon(1).
\end{align*}

\smallskip

\noindent{\bf Conclusion:} Sending $\varepsilon\to0$ and combining the
above estimates for $\mu$ 
in the cases whether $t_0$ and/or $s_0$ are positive or zero, we find that
\begin{align}\label{final}
\mu &\leq C\bigg(\frac1\gamma+\frac1{\eta^{\tilde q}}\bigg)+ CT\bigg(\eta\Delta
  t+\gamma\Delta t^2+\frac\gamma{\eta^{2q}} +\Gamma(\sigma,\delta)\Big(
  \frac1{\gamma}+\frac1{\eta^q}+\Delta t+\gamma\Delta x \Big)\bigg),\\
\intertext{when $\eta$ does not depend on $t$ then}
\mu &\leq C\bigg(\frac1\gamma+\frac1{\eta^{\tilde q}}\bigg)+ CT\bigg(\eta\Delta
  t+\frac1{\eta^q}+ \Gamma(\sigma,\delta)\Big(
  \frac1{\gamma}+\gamma\Delta x\Big) \bigg), \label{final2}
\intertext{and finally when $\eta$ does not depend on both $x$ and $t$ then}
\mu &\leq C\bigg(\frac1\gamma+\frac1{\eta^{\tilde q}}\bigg)+ CT\bigg(\eta\Delta
  t+\frac1{\eta^q}  + \frac{1}{\gamma} + \Gamma(\sigma,\delta)\gamma\Delta x \bigg). \label{final3}
\end{align}
Here $q=1=\tilde q$ if $K(u_0)<\infty$, otherwise $q=\frac
\sigma{2\sigma-1}$ and $\tilde q=\frac1{2\sigma-1}$ (when $\sigma\neq 1$).

\subsection{Proof of Theorem \ref{main:thm} when \texorpdfstring{$\sigma\in[0,1)$}.} 
In this case $\sigma\in(0,1)$, $K(u_0)<\infty$, $\Gamma(\sigma,\delta)=1$, and $q=1$ in
\eqref{final} since $u$ is Lipschitz in $t$ by Theorem \ref{wellpos_1} (d). From estimate \eqref{final} and our assumptions (note that $\Delta x\leq
\delta\leq1$),  we see that the optimal parameter values are
$\eta=\gamma$. This leads to the following bound
\begin{align*}
\mu\leq C\frac1\gamma+CT\bigg( \frac1\gamma + \gamma \Big(\Delta t+\Delta x\Big)\bigg).
\end{align*}
We conclude the proof of Theorem \ref{main:thm} (a) by taking $\gamma =
(T\wedge1)^{-\frac12}(\Delta t + \Delta
x)^{-\frac12}$ and then
adding the estimate from Lemma \ref{lem_delta}.  
\bigskip


\subsection{Proof of Theorem \ref{main:thm}  when \texorpdfstring{$\sigma\in(1,2)$}.}
In this case $\sigma\in(1,2)$, $\Gamma(\sigma,\delta)> 1$, and
  $q=\frac\sigma{2\sigma-1}$ and $\tilde q=\frac1{2\sigma-1}$ in
\eqref{final} since $u$ and $U_h$ are only H\"older
$\frac1\sigma$ in $t$ at $t=0$ by Theorem \ref{wellpos_1}
(c) and Theorem \ref{thm:existence-scheme} (d). The optimal
values for $\eta$ and $\gamma$ in \eqref{final} can be chosen
by balancing the principal terms. This leads to
$$\gamma=\min\Big\{((T\wedge1)\Delta x)^{-\frac12},((T\wedge1)\Delta t^2)^{-\frac12},\eta^{q}\Big\}\qquad\text{and}\qquad
\eta={((T\wedge1)\Delta t)^{-\frac1{1+\tilde q}}}.
$$
  Then $\frac1{1+\tilde q}=\frac{2\sigma-1}{2\sigma}$, $(T\wedge1)\eta\Delta t=\frac1{\eta^{\tilde q}}=((T\wedge1)\Delta t)^{\frac1{2\sigma}}$,
$\frac{\gamma}{\eta^{2q}}\leq\frac1{\eta^q}=((T\wedge1)\Delta t)^{\frac 12}$,
and by our assumptions (including $\delta,\Delta x,\Delta t\leq 1$),
\eqref{final} implies that  
\begin{align*}
\mu &\leq C\bigg(\frac1\gamma+\frac1{\eta^{\tilde q}}\bigg)+ CT\bigg(\eta\Delta
t 
+\Gamma(\sigma,\delta)\Big(
\frac1{\gamma}+\frac1{\eta^{q}}+\gamma\Delta x\Big) \bigg)\\
&\leq
  C(1+T)\bigg(\big((T\wedge1)\Delta
  t\big)^{\frac1{2\sigma}}+\Gamma(\sigma,\delta)\Big(\big((T\wedge1)\Delta
  x\big)^{\frac12}+\big((T\wedge1)\Delta t\big)^{\frac 12}\Big)\bigg). 
\end{align*}
We conclude the proof of Theorem \ref{main:thm} (b) by adding the
estimate from Lemma \ref{lem_delta}.  

\subsection{Proof of Theorem \ref{main:thm}  when \texorpdfstring{$\sigma=1$}.}

The proof is a combination of the proof of the case $\sigma\in(1,2)$
and the regularization argument of the proof of Theorem
\ref{thm:existence-scheme}\textit{(e)}. Let $u_0^{\tilde{\epsilon}}$
be the mollified initial data and $u^{\tilde{\epsilon}}$ and
$U^{\tilde{\epsilon}}_h$ be the corresponding solutions of
\eqref{eqn:main-perturbed-1} and \eqref{eq:the scheme} both with initial
condition $u_0^{\tilde{\epsilon}}$. Then we double
the variables by redefining $\Psi$ to be
\begin{align*} 
\Psi(t,x,s,y) = U^{\tilde{\epsilon}}_h(t,x)- u^{\tilde{\epsilon}}(s,y) -
  \phi(x,y) - \xi(t,s)  -\frac{\mu}{4T} (t+s) 
\end{align*}
where $\mu=\sup_{\mathcal{G}_h^N}\big(U_h^{\tilde{\epsilon}} - u^{\tilde{\epsilon}}\big)$ and
$\phi$ and $\xi$ are the same as before. As before, there exists a 
maximum point $(x_0,y_0,t_0,s_0)$ of $\Psi$ satisfying
\eqref{max_pt}--\eqref{esti_00}. By  Theorem \ref{wellpos_1}, $|u^{\tilde{\epsilon}}(t,y) -
u^{\tilde{\epsilon}}(s,y)| \leq K(u_0^{\tilde{\epsilon}})|t-s|$ for $K(u_0^{\tilde{\epsilon}})=C (1+ |\log
\tilde{\epsilon}|)$, and hence by \eqref{esti_00}
\begin{align}\label{esti_2_sigma1}
|t_0-s_0| \leq \frac{K(u_0^{\tilde{\epsilon}})}{\eta}. 
\end{align}    
At this point the proof continues as for the case $\sigma\in(1,2)$ but
with \eqref{esti_2_sigma1} replacing \eqref{esti_2}. If either $t_0=0$
or $s_0=0$ we use as before regularity to estimate $\mu$. E.g. 
if $s_0=0$, then since $\Psi (t,x,t,x)\leq \Psi (t_0,x_0,0,y_0)$,
\begin{align*}
&U_h^{\tilde{\epsilon}}(t,x)  - u^{\tilde{\epsilon}}(t,x) - \varepsilon|x|^2 -\frac{\mu}{2T}t\leq  U_h^{\tilde{\epsilon}}(t_0, x_0)- u_0^{\tilde{\epsilon}}(y_0)   
\leq C \Big( K(u_0^{\tilde{\epsilon}})t_0+ |x_0 - y_0|\Big)  
\leq C \Bigg( \frac{K(u_0^{\tilde{\epsilon}})^2}{\eta}+ \frac{1}{\gamma}\Bigg),  
\end{align*}
where we used \eqref{esti_1} and \eqref{esti_2_sigma1} for the last
inequality. We send $\varepsilon\to0$ and take the
supremum over $\mathcal{G}_h^N$ to find that  
\begin{align}\label{esti_t02_sigma1}
\mu \leq C \Bigg( \frac{K(u_0^{\tilde{\epsilon}})^2}{\eta}+ \frac{1}{\gamma}\Bigg).
\end{align}
The same bound holds also when $t_0=0$. When both $t_0>0$ and
$s_0>0$, the proof for $\sigma\in(1,2)$ is valid also for $\sigma=1$
up until and including the bound \eqref{sigma_bound}.
We add the estimates on $\mu$, \eqref{sigma_bound} and
\eqref{esti_t02_sigma1}, use estimates \eqref{esti_1} and \eqref{esti_2_sigma1}, and send
$\varepsilon \rightarrow 0$ (compare with \eqref{final}), to get
\begin{align}\label{final_sigma1}
\mu &\leq C\bigg(\frac1\gamma+\frac{K(u_0^{\tilde{\epsilon}})^2}{\eta}\bigg)
+ CT\bigg(\eta \Delta t+\gamma \Delta t^2+\frac{\gamma K(u_0^{\tilde{\epsilon}})^2}{\eta^2} +|\log\delta| \Big(\frac{1}{\gamma}+\frac{K(u_0^{\tilde{\epsilon}})}{\eta}+\Delta t+\gamma\Delta x \Big)\bigg). 
\end{align}
Taking optimal values of $\gamma$ and $\eta$ in \eqref{final_sigma1}
by balancing the principal terms, then leads to
$$ \gamma=\min\Big\{((T\wedge1)\Delta x)^{-\frac12},((T\wedge1)\Delta t^2)^{-\frac12},\frac{\eta}{K(u_0^{\tilde{\epsilon}})}\Big\}\qquad\text{and}\qquad
\eta=\frac{K(u_0^{\tilde{\epsilon}})}{((T\wedge1)\Delta t)^{1/2}},$$
and hence
\begin{align*}
\big(U_h^{\tilde{\epsilon}} - u^{\tilde{\epsilon}}\big)\leq  \ \mu & \  \leq C(1+T)(T\wedge1)^{\frac12}\bigg( \Delta x^{\frac12} + |\log \tilde\epsilon|\Delta t^{\frac12}+ |\log\delta|\big( \Delta x^{\frac12} + \Delta t^{\frac12}\big)   \bigg).
\end{align*}
A bound for $( u^{\tilde{\epsilon}}-U_h^{\tilde{\epsilon}} )$ can be found by interchanging  the roles of $u^{\tilde{\epsilon}}$ and $U_h^{\tilde{\epsilon}}$. By comparison, Theorems \ref{wellpos_1} \textit{(a)} and
\ref{thm:existence-scheme}, \textit{(a)}, $|U_h-
U_h^{\tilde{\epsilon}}|,|u^{\tilde{\epsilon}} -
u^{\delta}|\leq |u^{\tilde{\epsilon}}_0 -
u_0|\leq C\tilde\epsilon$, and then
\begin{align*}
|U_h(t,x) - u^{\delta}(t,x)|  & \leq |U_h(t,x)-
                                         U_h^{\tilde{\epsilon}}(t,x)|+
                                         |U_h^{\tilde{\epsilon}}(t,x)
                                         - u^{\tilde{\epsilon}}(t,x)|
                                         + |u^{\tilde{\epsilon}}(t,x)
                                         - u^{\delta}(t,x)|  \\
                              & \leq 2\tilde\epsilon +C(1+T)(T\wedge1)^{\frac12}\bigg( \Delta x^{\frac12} + |\log \tilde\epsilon|\Delta t^{\frac12}+ |\log\delta|\big( \Delta x^{\frac12} + \Delta t^{\frac12}\big)   \bigg).
\end{align*}
The proof of Theorem \ref{main:thm} \textit{(b)} for
$\sigma=1$ is complete by taking $\tilde{\epsilon} = \Delta t$ and adding the estimate of Lemma \ref{lem_delta}.

\subsection{Proof of Theorem \ref{main:thm2} (a).}
We only do the case $\sigma\in(1,2)$. The case
  $\sigma=1$ follows in a similar way, cf.  proof of Theorem
  \ref{main:thm} for $\sigma=1$, and the case $\sigma\in[0,1)$ follows
  directly from Theorem \ref{main:thm}. Now  $\Gamma(\sigma,\delta)> 1$, and
  $q=\frac\sigma{2\sigma-1}$ and $\tilde q=\frac1{2\sigma-1}$ in
\eqref{final} since $u$ and $U_h$ are only H\"older
$\frac1\sigma$ in $t$ at $t=0$ by Theorem \ref{wellpos_1}
(c) and Theorem \ref{thm:existence-scheme} (d).
 Note that when $\Delta t\leq \Delta x$, $\gamma\leq\eta^{q}$, and
$\gamma\geq1$ -- then
$\frac{\gamma}{\eta^{2q}}\leq \frac1{\eta^q}$, $\frac1{\eta^q}\leq
\frac1\gamma$ and $\Delta t\leq \gamma\Delta x$. By our assumptions, both 
\eqref{final} with $\Delta t\leq \Delta x$ (and then
$(1\leq)\,\gamma\leq\eta^{q}$,
see below!) and \eqref{final2} implies that 
\begin{align*}
\mu &\leq C\bigg(\frac1\gamma+\frac1{\eta^{\tilde q}}\bigg)+ CT\bigg(\eta\Delta
t+\frac1{\eta^{q}} 
+\Gamma(\sigma,\delta)\Big(
  \frac1{\gamma}+\gamma\Delta x\Big) \bigg). 
\end{align*}
We conclude the proof of  Theorem \ref{main:thm2} (a) by taking taking
$\eta=((T\wedge1)\Delta t)^{-\frac1{1+\tilde q}}$, $\gamma=((T\wedge1)\Delta
x)^{-\frac12}$, and then adding the estimate from Lemma \ref{lem_delta}.

\subsection{Proof of Theorem \ref{main:thm2} (b).}
In this case $\sigma\in(1,2)$, $\Gamma(\sigma,\delta)> 1$, and $q=1$ in
\eqref{final}  and \eqref{final2} since $u$ and $U_h$ are Lipschitz
in $t$ at $t=0$ by Theorem \ref{wellpos_1} 
(d) and Theorem \ref{thm:existence-scheme} (e). By our assumptions (note that
$\Delta x\leq \delta\leq1$), both \eqref{final} with $\Delta t\leq
\Delta x$ (and then $\gamma\leq\eta$, see below!) and \eqref{final2}
implies that
\begin{align*}
\mu &\leq C\bigg(\frac1\gamma+\frac1{\eta}\bigg)+ CT\bigg(\eta\Delta
t+\frac1{\eta} 
+\Gamma(\sigma,\delta)\Big(
  \frac1{\gamma}+\gamma\Delta x\Big) \bigg). 
\end{align*}
We conclude the proof of Theorem \ref{main:thm2} (b)  by taking $\eta=((T\wedge1)\Delta t)^{-\frac12}$,
$\gamma=((T\wedge1)\Delta x)^{-\frac12}$, and then adding the
estimate from Lemma \ref{lem_delta}.

\subsection{Proof of Theorem \ref{main:thm3} (a).}
Again we only do the case $\sigma\in(1,2)$. The case
  $\sigma=1$ follows in a similar way, cf.  proof of Theorem
  \ref{main:thm} for $\sigma=1$, and the case $\sigma\in[0,1)$ follows
  directly from Theorem \ref{main:thm}. Again  $\Gamma(\sigma,\delta)> 1$, and
  $q=\frac\sigma{2\sigma-1}$ and $\tilde q=\frac1{2\sigma-1}$ in
\eqref{final3}. 
We conclude the proof by taking taking
$\eta=((T\wedge1)\Delta t)^{-\frac1{1+\tilde q}}$,
$\gamma=\Big((T\wedge1)\Gamma(\sigma,\delta)\Delta x\Big)^{-\frac12}$, and
then adding the estimate from Lemma \ref{lem_delta}.

\section{On suboptimal rates for general monotone
  schemes} \label{sec:higher-order}
A close inspection of our proofs shows
that our methods can handle a large class of monotone approximations of
\eqref{eqn:main}-\eqref{int-cond} that allow for truncation errors
involving derivatives of at most order two. In most numerical
approximations it is possible to use {\em suboptimal} 
truncation errors that satisfy this condition. 
The resulting error estimates will not be optimal in general, but at
this point there are no alternative methods to get error
estimates for general  Isaacs equations. 

We illustrate this approach by proving 
suboptimal error estimates for an improved version of our previous
scheme. The idea is to compensate for the truncation of the nonlocal operator
$\mathcal{I}^{\A,\B}$ by a vanishing local diffusion. To do so, first note
that $\mathcal{I}^{\A,\B}[\phi] 
= \mathcal{I}^{\A,\B,\delta}[\phi]+
\mathcal{I}^{\A,\B}_{\delta}[\phi]$ where
$\mathcal{I}^{\A,\B,\delta}[\phi]$ is defined in \eqref{viscosity:term2} and  
\begin{align*}
 \mathcal{I}^{\A,\B}_{\delta}[\phi](t,x)= \int_{|z|\le \delta} \big( \phi(t,x+\eta^{\A,\B}(t,x; z)) - \phi(t, x) -\eta^{\A,\B}(t,x;z)\cdot \nabla_x\phi(t,x)\big) \, \nu(dz).
\end{align*}
By Taylor expansion we see that we can approximate this term by the
local term (cf. e.g. \cite{JKL08})
\begin{align*}
  \textrm{tr}\Big[a_\delta^{\A,\B}(t,x)D^2
\phi(t,x)\Big]\qquad\text{with}\qquad a_\delta^{\A,\B}(t,x)=\frac 12
  \int_{|z|\le\delta} \eta^{\A,\B}(t,x;z)
  \eta^{\A,\B}(t,x;z)^T\nu(dz)
  \end{align*}
and the error is $C\|D^3\phi\|_\infty\int_{|z|\le\delta}
|\eta^{\A,\B}(t,x;z)|^3\nu(dz)\leq C \|D^3\phi\|_\infty
\delta^{3-\sigma}$ in view of \ref{A3} and \ref{A5}. Next we take a
monotone finite difference approximation $\mathcal{L}^{\A,\B}_{\delta,
  h}[\phi]$ of this local term with error $K \|a_\delta^{\A,\B}\|_0\| D^4 \phi\|_0
(\Delta x)^2\leq K \delta^{2-\sigma} (\Delta x)^2 \| D^4 \phi\|_0$.
Note that to ensure this rate, we have to assume e.g. that
$a_\delta^{\alpha,\beta}$ is diagonally dominant. Under this
assumption, the (wide stencil) schemes of Kushner \cite{Ku:Book},
Bonnans-Zidani \cite{BZ} or Krylov \cite{Krylov:2005lj} would give this
error. Combining these results, we conclude that  $\mathcal{L}^{\A,\B}_{\delta,
  h}$ is an approximation of $\I_\delta^{\alpha,\beta}$ with error
\begin{align*} 
\big|\mathcal{I}^{\A,\B}_{\delta}[\phi]-\mathcal{L}^{\A,\B}_{\delta,
  h}[ \phi] \big|\leq C\Big( ||D^3 \phi||_0\delta^{3-\sigma}+||D^4 \phi||_0\Delta x^2 \delta^{2-\sigma} \Big) .
\end{align*}

Now we discretize equation \eqref{eqn:main} as in Section
\ref{sec:main} except that we do not throw away the
$\I^{\alpha,\beta}_\delta$-term but rather approximate it by
$\mathcal{L}^{\A,\B}_{\delta, h}$. The resulting 
semidiscrete approximation is then (compare with \eqref{eqn:main-semidiscrete})
 \begin{align}  \label{eqn:main-semidiscrete2}
 & u_t + \inf_{\A\in \mathcal{A}}\sup_{\B\in\mathcal{B}} \left\{ -f^{\A,\B}(t,x) +c^{\A,\B}(t,x) u(t,x) -\mathcal{D}_h^{\A, \B,\delta}[u](t,x)- \mathcal{L}_{\delta,h}^{\A,\B}[u](t,x)-\mathcal{J}_h^{\A,\B,\delta}[u](t,x) \right\}  = 0. 
\end{align} 
In view of the discussion above and in Section \ref{sec:main}, the
truncation error for this scheme is
\begin{align*}
  \nonumber
  &E:=\big\|b^{\alpha,\beta}\cdot
  \nabla \phi+\mathcal{I}^{\A,\B}[\phi]-(\mathcal{D}_h^{\A, \B,\delta}+\mathcal{L}^{\A,\B}_{\delta,
    h}+\mathcal{J}_h^{\A,\B,\delta})[ \phi] \big\|_0\\ 
  &\leq C\bigg( \Delta x\, \Gamma(\sigma,\delta) ||D^2 \phi||_0+||D^3 \phi||_0\delta^{3-\sigma}+||D^4 \phi||_0\Delta x^2 \delta^{2-\sigma}+||D^2\phi||_0 \frac{\Delta x^2}{\delta^{\sigma}} \bigg) .
\end{align*}
For $\sigma\in[0,1)$ or $\sigma=1$, the optimal choice of $\delta$ is
  $\delta=\Delta x$ and then $E=O(\Delta x)$ or
  $E=O(\Delta x|\ln\Delta x|)$ as in the previous section. But when
  $\sigma\in(1,2)$, then the two first terms in the bound on $E$
  dominate and the optimal choice is $\delta=\Delta x^{\frac12}$. The
  corresponding error $E=O(\Delta x^{\frac{3-\sigma}{2}})$ is
  better than the  (optimal) truncation error
  $O(\Delta x^{2-\sigma})$ from Section \ref{sec:main} (see Remark
  \ref{optd-rem}), especially when $\sigma\approx 2$. 
To find a useful suboptimal bound, note that 
$|\mathcal{L}^{\A,\B}_{\delta, 
    h}[\phi]|\leq C\|a_\delta^{\alpha,\beta}\|_0\|D^2\phi\|_0\leq
C\delta^{2-\delta}\|D^2\phi\|_0$ and
$|\mathcal{L}^{\A,\B}_{\delta}[\phi]|\leq
C\delta^{2-\sigma}\|D^2\phi\|_0$, and then
\begin{align*}
  \tilde E&:= \big|b_\delta^{\alpha,\beta}\cdot
  \nabla \phi-\mathcal{D}_h^{\A, \B,\delta}[\phi]\big|+|\mathcal{J}^{\A,\B,\delta}[\phi]-\mathcal{J}_h^{\A,\B,\delta})[\phi]\big|+\big|\mathcal{L}^{\A,\B}_{\delta}[\phi]\big|+\big|\mathcal{L}^{\A,\B}_{\delta,
    h}[\phi]\big|\\
  &\ \leq C||D^2 \phi||_0\bigg( \Delta x\, \Gamma(\sigma,\delta) + \frac{\Delta x^2}{\delta^{\sigma}} +\delta^{2-\sigma}\bigg). 
\end{align*}
This is same estimate that was optimal for the scheme in Section \ref{sec:main}.

A fully discrete scheme is then obtained by discretizing \eqref{eqn:main-semidiscrete2} in time as in
\eqref{eq:the scheme}. For simplicity, we only consider an implicit
scheme here:
\begin{align}\label{eq:the scheme_diffusion}
  U_{\bj} ^n = &\ U_\bj^{n-1}-\Delta t  \inf_{\A\in
    \mathcal{A}}\sup_{\B\in\mathcal{B}} \left\{- f^{\A,\B, n}_\bj +
  c_\bj^{\A, \B,n} U_\bj^{n}-  \mathcal{D}_h^{\A,
    \B,\delta}[U]_\bj^{n}
  -\mathcal{J}_h^{\A,\B,\delta}[U]_\bj^{n}-\mathcal{L}^{\A,\B}_{\delta,
    h}[U]^{n}_{\bj} \right\}.
 \end{align}
We have the following result.

\begin{thm} Assume $\mathcal{L}_{\delta,h}^{\alpha,\beta}$ is as
  explained above and $\eta^{\A,\B}$ does not depend on $x, t$. Then Theorem \ref{main:thm3}
  remains true when the scheme \eqref{eq:the scheme} is replaced by the scheme
  \eqref{eq:the scheme_diffusion}. 
\end{thm}

\begin{proof}[Outline of proof]
We follow the proof of Section \ref{sec:dbl} without doing the truncation
step in Section \ref{sec:red} first. The idea is to estimate
separately the terms $\mathcal{L}^{\A,\B}_{\delta}[\phi]$ and
$\mathcal{L}^{\A,\B}_{\delta, h}[\phi]$.  By the discussion above and
the definition of the test function $\phi$, both terms are bounded by the
vanishing viscosity like bound
$C(\gamma+\varepsilon)\delta^{2-\sigma}$, and in the proof this term
would appear as new term $I_5$ on the right hand side of
\eqref{diff}. Continuing the proof,  the bound on $\mu$ in \eqref{final3} will have this additional term i.e.
\begin{align*}
\mu &\leq C\bigg(\frac1\gamma+\frac1{\eta^{\tilde q}}\bigg)+ CT\bigg(\eta\Delta
  t+\frac1{\eta^q}+ \Gamma(\sigma,\delta)\gamma\Delta x +\gamma\delta^{2-\sigma}\bigg).
\end{align*}
To conclude the (same) error estimates, we now have to modify the
choice of $\gamma$ and take
$$\gamma=\min\Big(\big((T\wedge
1)\Gamma(\sigma,\delta)\Delta x\big)^{-\frac12}
,(T\wedge1)^{-\frac12}\delta^{-(1-\frac\sigma2)}\Big)$$
which  leads to the bound $\mu\leq \text{$\Delta t$-term}
+C(1+T)(T\wedge1)^{\frac12}\big(\Gamma(\sigma,\delta)^{\frac12}\Delta x^{\frac12}+\delta^{1-\frac\sigma2}\big)$.
\vspace*{1cm}
 \end{proof}
  \begin{rem}
   (a) If $\eta$ does not depend on $(x,t)$, then our approach gives error bounds for arbitrary monotone schemes that admit possibly suboptimal error expansion involving no higher order derivatives than order $2$. 
   
 \medskip  
 \noindent (b) If $\eta$ depends on $x$, then the results will not be
 so good. Redoing the proof outlined above, we have to replace
 \eqref{final3} by \eqref{final} or \eqref{final2} which contain an additional
$O\big(\Gamma(\sigma,\delta)\frac{1}{\gamma}\big)$ term.
 To get the final error bound, we now have to take a $\gamma$ that minimize
 $$\Gamma(\sigma,\delta) \Big( \frac{1}{\gamma}+ \gamma \Delta
 x\Big) + \gamma \delta^{2-\sigma}.$$
This leads to
 $\gamma = \min \Big( \Delta x^{-1/2},
 \frac{\Gamma(\sigma,\delta)^{1/2}}{\delta^{\frac{2-\sigma}{2}}} \Big)
 = \min(\Delta x^{-1/2}, \delta^{-1/2} )= \delta^{-1/2}$
 since $\Delta x \leq \delta<1$, and then
   $$\mu \leq \dots + C\Big(\delta^{1-\sigma}(\delta^{1/2} + \Delta
   x^{1/2} ) + \delta^{-1/2} \delta^{2-\sigma}\Big)= \dots + C \Big(
   \delta^{1-\sigma} \Delta x^{1/2} +
   \delta^{\frac{3}{2}-\sigma}\Big). $$ This error bound is worse than
   before, and only valid for $\sigma \leq \frac 32$. 
 
 \medskip
 \noindent (c) A possible way to obtain general (suboptimal) results
 when $\eta$ depends on $(x,t)$, is via continuous dependence results
 like in \cite{JKL08}. But now such results are also needed for the
 scheme. Obtaining such results can be challenging in general and will
 not be considered here.
 \end{rem}
 


\section*{Acknowledgments} E. R. Jakobsen is supported by the
Toppforsk (research excellence) project Waves and Nonlinear Phenomena
(WaNP), grant no. 250070 from the Research Council of Norway.  I. H. Biswas acknowledges the support received from INSA via {\it INSA Young Scientist Project}. In addition, Indranil Chowdhury is supported by the research fellowship of {\it Department of Atomic Energy, India}.




\begin{thebibliography}{}


\bibitem{BI:P07}
G.~Barles and C.~Imbert.
\newblock Second-Order Elliptic Integro-Differential Equations:
  Viscosity Solutions' Theory Revisited.
\newblock {\em Ann. Inst. H. PoincarŽ Anal. Non LinŽaire} 25 (2008), no. 3, 567--585.



\bibitem{BJ2002}
G.~Barles and E.~R. Jakobsen.
\newblock On the convergence rate of approximation schemes for Hamilton-Jacobi-Bellman equations. \newblock{\it M2AN Math. Model. Numer. Anal.} 36 (2002), no. 1, 33--54.

\bibitem{BJ2005}
G.~Barles and E.~R. Jakobsen.
\newblock{ Error bounds for monotone approximation schemes for Hamilton-Jacobi-Bellman equations}. 
\newblock{\em SIAM J. Numer. Anal.} 43 (2005), no. 2, 540--558 (electronic)


\bibitem{Barles:2006jf}
{G.~Barles and E.~R. Jakobsen},
\newblock Error bounds for monotone approximation schemes for parabolic Hamilton-Jacobi-Bellman equations,
\newblock{\it Math. Comp.}, 76 (2007), ~1861--1893.




\bibitem{Barles:1991tc}
{ G.~Barles and P.~E. Souganidis},
\newblock{ Convergence of approximation schemes for fully nonlinear second order equations},
\newblock{\it Asymptotic Anal.}, 4 (1991), ~271--283.


\bibitem{Biswas2012}
 I.~ H.~Biswas.
 \newblock  On zero-sum stochastic differential games with jump-diffusion driven state: a viscosity solution framework.
 \newblock{\em SIAM J. Control Optim.} 50 (2012), no. 4, 1823--1858.



\bibitem{BJK2}
I.H. Biswas, E.R. Jakobsen and K.H. Karlsen.
\newblock Error estimates for a class of finite difference-quadrature schemes for fully nonlinear degenerate parabolic integro-PDEs.
\newblock{\em  J. Hyperbolic Differ. Equ.} 5 (2008), no. 1, 187--219. 


\bibitem{BJK1}
I.H. Biswas, E.R. Jakobsen and K.H. Karlsen.
\newblock Difference-quadrature schemes for nonlinear degenerate parabolic integro-PDE. 
\newblock{\em SIAM J. Numer. Anal.} 48 (2010), no. 3, 1110--1135. 


\bibitem{BJK_SW} 
I. H. Biswas, E. R. Jakobsen, and K. H. Karlsen.
\newblock  Viscosity solutions for a system of integro-PDEs and connections to optimal switching and control of jump-diffusion processes.
\newblock {\em   Appl. Math. Optim.} 62 (2010), no. 1, 47--80.
 
\bibitem{Zidani:2006}
F.~ Bonnans, S. Maroso and H. Zidani.
\newblock Error estimates for stochastic differential game: the adverse stopping case Frédéric Bonnans, 
\newblock{\em IMA Journal of Numerical Analysis.} 28(2006), 188--212.  
 
\bibitem{BZ}
J. F. Bonnans and H. Zidani.
\newblock Consistency of generalized finite difference schemes for the
stochastic HJB equation.
\newblock {\em SIAM J. Numer. Anal.} 41 (2003), no. 3, 1008--1021.
 
 
 \bibitem{Caffarelii:2008gf}
 L. A. Caffarelli and P. E. Souganidis.
 \newblock A rate of convergence for monotone finite difference approximations to fully nonlinear, uniformly elliptic PDEs.
 \newblock{\em Comm. Pure Appl. Math. }  61 (2008), no. 1, 1--17. 

\bibitem{CDF}
  I. Capuzzo-Dolcetta and M. Falcone.
  \newblock
Discrete dynamic programming and viscosity solutions of the Bellman equation. 
\newblock{\em Ann. Inst. H. Poincare Anal. Non Lineaire} 6 (1989),
suppl., 161--183. 
 
 \bibitem{CJ2011}
 S. Cifani and E. R. Jakobsen.
   \newblock{ On numerical methods and error estimates for degenerate fractional convection-diffusion equations.}
   \newblock{\it Numer. Math.} 127(2014), no. 3, 447--483. 
 
\bibitem{CTbook}
R. Cont and P. Tankov.
\newblock Financial Modelling with Jump Processes.
\newblock {\it Chapman and Hall/CRC}, 2003.

\bibitem{CIL}
  M. G. Crandall, H. Ishii, and P. L. Lions.
  \newblock User's guide to viscosity solutions of second order
  partial differential equations.
  \newblock {\em Bull. Amer. Math. Soc. (N.S.)} 27 (1992), no. 1, 1--67.

\bibitem{CL:2approx}
{ M.~G. Crandall and P.-L. Lions},
{ Two approximations of solutions of Hamilton-Jacobi equations},
{\it Math. Comp.}, 43 (1984), ~1--19.




 
 \bibitem{Evans:1984fv}
 L.~C.~Evans and P.~E.~Souganidis.
\newblock  Differential games and representation formulas solutions for Hamilton-Jacobi-Isaacs equations,
\newblock{\it Indiana Univ. Math. J.}, 33(1984), ~ 773--797.
 
 \bibitem{Fleming:1989cg}
 W. H.~Fleming and P.E.~Souganidis,
\newblock On the existence of value functions of two player, zero-sum stochastic differential Games,
\newblock{\it Indiana University Mathematics Journal}, 38 (1989), ~ 293--314.

\bibitem{J2002}
 E.~R. Jakobsen. 
 \newblock On error bounds for approximation schemes for non-convex degenerate elliptic equations. 
 \newblock{BIT} 44 (2004), no. 2, 269--285. 

 
 \bibitem{J2006}
 E. R. Jakobsen. On error bounds for monotone approximation schemes for multi-dimensional Isaacs equations.
 \newblock{ \it Asymptotic Analysis} 49 (2006), no. 3, 249--273.


\bibitem{Jakobsen:2005jy}
E.~R. Jakobsen and K.~H. Karlsen.
\newblock Continuous dependence estimates for viscosity solutions of
  integro-{PDE}s.
\newblock {\em J. Differential Equations}, 212 (2005), no. 2, 278--318.



\bibitem{JKL08}
E. R. Jakobsen, K. H. Karlsen, and C. La Chioma.
\newblock Error estimates for approximate solutions to Bellman
equations associated with controlled jump-diffusions. 
\newblock {\em Numer. Math.} 110 (2008), no. 2, 221--255. 






\bibitem{Krylov:1997nf}
{ N.~V. Krylov}.
\newblock On the rate of convergence of finite-difference approximations for Bellman's equations,
\newblock {\it Algebra i Analiz}, 9 (1997), ~245--256.


\bibitem{Krylov:2000yy}
{ N.~V. Krylov},
\newblock 
On the rate of convergence of finite-difference approximations for Bellman's 
equations with variable coefficients,
\newblock{\it Probab. Theory Related Fields}, 117 (2000), ~1--16.


\bibitem{Krylov:2005lj}
{ N.~V. Krylov},
\newblock 
The rate of convergence of finite-difference approximations for Bellman equations with Lipschitz coefficients,
\newblock {\it Appl. Math. Optim.}, 52 (2005), ~365--399.




\bibitem{krylov:2015}
N. V.  Krylov.
 \newblock On the rate of convergence of finite-difference
  approximations for elliptic Isaacs equations in smooth
  domains. 
  \newblock{\it Comm. Partial Differential Equations} 40 (2015), no. 8,
  1393--1407.


\bibitem{Ku:Book}
H.~J. Kushner and P.~G. Dupuis.
\newblock {\em Numerical methods for stochastic control problems in continuous
time}.
\newblock Springer-Verlag, New York, 1992.
  
  \bibitem{Lions:2010hs}
{ P.-L. Lions and B.~Mercier},
 \newblock Approximation num\'erique \'equations de Hamilton-Jacobi-Bellman,
{\it RAIRO Anal. Num\'er.}, 14 (1980), ~369--393.




\bibitem{souganidis1985}
P.~E.~Souganidis.
 \newblock{Approximation schemes for viscosity solutions of Hamilton-Jacobi equations}. 
 \newblock{\it J. Differential Equations} 59 (1985), no. 1, 1--43.


\bibitem{T2015}
 O. Turanova.
 \newblock Error estimates for approximations of nonhomogeneous
 nonlinear uniformly elliptic equations. 
 \newblock{\em Calc. Var. Partial
 Differential Equations} 54 (2015), no. 3, 2939--2983. 

\bibitem{T2015a}
 O. Turanova. 
\newblock Error estimates for approximations of nonlinear uniformly parabolic equations.
\newblock{\em NoDEA Nonlinear Differential Equations Appl.} 22 (2015), no. 3, 345--389.
\end{thebibliography}
\end{document}